\documentclass[11pt]{article}
\usepackage{amsmath,amssymb,stmaryrd,mathrsfs,lmodern,comment,centernot,tikz,tikz-cd,mathtools}
\usetikzlibrary{arrows}
\usepackage[section]{placeins}
\usepackage{adjustbox}
\usepackage{enumitem}
\usepackage{authblk}
\usepackage[normalem]{ulem}
\usepackage{xcolor}
\usepackage[margin=1in]{geometry}
\usepackage{setspace}

\newcommand{\limp}{\longrightarrow}
\newcommand{\liml}{\mathsf{Lim}_\mathscr{L}}
\newcommand{\tl}[1]{\mathrm{TL}_{#1}}
\DeclareMathOperator{\OC}{\mathsf{OrCl}}

\usepackage{amsthm}
\theoremstyle{definition}
\newtheorem{theorem}{Theorem}[section]
\newtheorem{lemma}[theorem]{Lemma}
\newtheorem{corollary}[theorem]{Corollary}
\newtheorem{definition}[theorem]{Definition}
\newtheorem{remark}[theorem]{Remark}
\newtheorem{example}[theorem]{Example}

\usepackage[backref=page,colorlinks=true,allcolors=blue]{hyperref}

\title{Lindenbaum-type Logical Structures\footnote{The final version of the article has been submitted to Logica Universalis.}}
  \author[1]{Sayantan Roy}
  \author[2]{Sankha S.\ Basu }
  \author[3]{Mihir K.\ Chakraborty}
  \date{July 22, 2021}
  \affil[1,2]{Dept.\ of Mathematics\\
  Indraprastha Institute of Information Technology-Delhi\\
  New Delhi, India.}
  \affil[3]{School of Cognitive Science\\Jadavpur University\\ Kolkata, India.}
  \date{September 04, 2022}
  
\begin{document}
  \maketitle
[The first author wishes to dedicate this paper to his grandparents, especially his grandfathers.]

\begin{abstract}
In this paper, we study some classes of logical structures from the universal logic standpoint, viz., those of the Tarski- and the Lindenbaum-types. The characterization theorems for the Tarski- and two of the four different Lindenbaum-type logical structures have been proved as well. The separations between the five classes of logical structures, viz., the four Lindenbaum-types and the Tarski-type have been established via examples. Finally, we study the logical structures that are of both Tarski- and a Lindenbaum-type, show their separations, and end with characterization, adequacy, minimality, and representation theorems for one of the Tarski-Lindenbaum-type logical structures.
\end{abstract}


\tableofcontents
\section{Introduction}\label{sec:Intro}
The field of logic has seen exponential growth since the turn of the twentieth century. This perhaps is, at least partially, due to logic and mathematics getting involved with each other. There was also an unprecedented upsurge in the number of logics, which continues to the present day. Therefore, a need was felt for a unifying notion of logic. This was also necessary for establishing logic as a bonafide and independent branch of mathematics, described by Bourbaki in \cite{Bourbaki1950} as the study of mathematical structures. Perhaps the first progress towards this end was made by Tarski in the early twentieth century. His work was later extended by the Polish logicians, and this is known as the theory of the \emph{consequence operator} (see Definition \ref{def:Tarski-type} for a version of this).

The modern notion of \emph{universal logic}, with the aim of providing a general theory of logics in the same way as universal algebra is a general study of algebraic structures, was introduced by B\'eziau in \cite{Beziau1994}. In particular, just as universal algebra does not propose that there is one algebra that is paramount or that there are some absolute laws that govern the realm of all algebras, universal logic neither asserts the existence of one universal logic nor the existence of any set of universal laws to be satisfied by all logics. In fact, as clarified by B\'eziau in \cite{Beziau2006}, ``\textellipsis from the viewpoint of universal logic, the existence of one universal logic is not even possible, and this is a result that can easily be shown.'' This is the sense in which we use the term `universal logic' in this paper (there have been other uses of the term, such as the one in \cite{Brady2006}, but these are not connected to the present study).

Fundamental to the study of universal logic is the notion of a \emph{logical structure} \cite{Beziau1994} (see Definition \ref{def:logical_structure}) much like the notion of an algebraic structure in universal algebra.

In this paper, we study two classes of logical structures, viz., logical structures of the Tarski- and of the Lindenbaum-types. The Tarski-type (see Definition \ref{def:Tarski-type}) refers to a single type of logical structures and has been widely studied (referred to as \emph{normal logics} in \cite{Beziau2001}). On the other hand, there are four possible definitions of a Lindenbaum-type logical structure (see Definition \ref{def:Lind-type}) based on the four versions of Lindenbaum's law discussed in \cite{Beziau1999}. In the present article, we have built upon and extended the study of the Lindenbaum-type logical structures initiated by B\'eziau in \cite{Beziau1999}. Questions regarding the possible relationships between these five classes (the Tarski- and the four Lindenbaum-types) have also been answered. Subsequently, the logical structures that are of both Tarski- and a Lindenbaum-type have been discussed. Furthermore, we have answered the questions regarding the relationships between the four Tarski-Lindenbaum-type logical structures. Central to the development of the Lindenbaum-type logical structures, is the notion of \emph{$\alpha$-saturated sets} (see Definition \ref{def:alpha-sat}). This is not a new concept and is defined as such, e.g., in \cite{Loparic_daCosta1984}; these are also defined as \emph{$a$-excessive theories} in \cite{Beziau1994}. We study these here in full generality and show that in case of known logics, this notion corresponds nicely to familiar ones. For example, it has been proved in Section \ref{sec:Lind} that the $\alpha$-saturated sets of a certain kind are precisely the implication-saturated sets described by Batens in \cite{Batens1980}.  

A central problem in developing any theory of structures is the characterization of the structures with certain special properties. For an analogy, one can think of the fundamental theorem of finite abelian groups. Likewise, in the theory of logical structures, a central problem is to characterize the logical structures possessing certain properties, chosen according to the purpose and goal for the logics under consideration. This is important because the logician, armed with such knowledge, can then better determine the applicability of a particular logic to a given situation (see \cite{Beziau2006} for a more detailed discussion on this). We believe that the characterization results obtained here for some of the aforementioned types of logical structures are among our major contributions in this article.

\subsection{Outline of the paper}\label{subsec:Intro/Outline}

We begin with some basic definitions in Subsection \ref{subsec:Intro/Fund_Not}. Next, the Tarski-type logical structures are discussed in Section \ref{sec:Tarski}. This includes a characterization theorem for these. Section \ref{sec:Sat_MaxNontriv} is split into two subsections and is devoted to the study of saturated and maximal nontrivial sets. Subsection \ref{subsec:SatSets} consists of a discussion of the concepts of saturated and $\alpha$-saturated sets. In Subsection \ref{subsec:MaxNontrivSets}, we discuss the maximal nontrivial sets. This is where we establish a relationship between the saturated and the maximal nontrivial sets and a generalized version of the implication-saturated sets that were introduced in \cite{Batens1980}. Section \ref{sec:Lind} starts with the definitions of the four Lindenbaum-type logical structures. As mentioned earlier, the study of the Lindenbaum-type logical structures is built upon the one in \cite{Beziau1999}, so here we map the terminologies from the latter to the ones introduced here. Then we split the subsequent discussion into four subsections - Subsection \ref{subsec:LindII} consists of a study of the Lindenbaum-II-type logical structures, and Subsection \ref{subsec:LindIV} is devoted to the Lindenbaum-IV-type logical structures, including a characterization theorem for these. The Lindenbaum-I- and Lindenbaum-III-type logical structures are considered in Subsection \ref{subsec:LindI/III}; this also includes a characterization theorem for the Lindenbaum-III-type logical structures. We discuss the relationships and separations between the four Lindenbaum-types in Subsection \ref{subsec:Lind-rel}. Next, in Section \ref{sec:T_R_L}, the possible relationships between the classes of logical structures of the Tarski- and the various Lindenbaum-types are investigated. We then proceed to study the logical structures that are of both Tarski- and a Lindenbaum-type in Section \ref{sec:Tar-Lind}. Here, we also expound on the relationships and separations between the resulting four Tarski-Lindenbaum-type logical structures. Section \ref{sec:TL-4} is devoted to the characterization, adequacy, minimality, and representation theorems for logical structures that are, at the same time, of Tarski- and Lindenbaum-IV/II-type. Finally, in the concluding section \ref{sec:Concl}, some directions for future research are indicated.

\subsection{Fundamental notions}\label{subsec:Intro/Fund_Not}

\begin{definition}\label{def:logical_structure}
A \emph{logical structure} is a pair $(\mathscr{L},\vdash)$, where $\mathscr{L}$ is a set and $\vdash\,\subseteq\mathcal{P}(\mathscr{L})\times \mathscr{L}$. Here $\mathcal{P}(\mathscr{L})$ denotes the power set of $\mathscr{L}$.
\end{definition}

In this paper, we will always assume that $\mathscr{L}$ and $\vdash$ are nonempty.

\begin{remark}\label{rem:logicVlog_str}
As mentioned earlier, logical structures were defined in \cite{Beziau1994}. However, although the terms `logic' and `logical structure' have been considered synonymous in \cite{Beziau1994}, we distinguish between these here. The term logic is reserved for referring to those pairs $(\mathscr{L},\vdash)$, where $\mathscr{L}$ is a formula algebra of some language over a set of variables. We also do not use the term `formula' to refer to the elements of $\mathscr{L}$. This is mainly to avoid introducing any potentially psychologically binding intuition about the elements of $\mathscr{L}$ (see \cite[Section 3.2]{Beziau1994} for further discussion on this matter).

It is, however, easy to see that any logic $(\mathscr{L},\vdash)$ has an underlying logical structure, the one obtained by ignoring the nature of, and any relationship between, the elements of $\mathscr{L}$. The underlying logical structure of a logic $(\mathscr{L},\vdash)$ will be referred to as such, and denoted by $(\lvert\mathscr{L}\rvert,\vdash)$. To avoid any confusion, it will be clearly indicated at the outset of every discussion whether $(\mathscr{L},\vdash)$ is a logic or a logical structure.
\end{remark} 

Let $(\mathscr{L},\vdash)$ be a logical structure and $\Gamma\subseteq\mathscr{L}$. Then $C_{\vdash}^{\mathscr{L}}(\Gamma)$ denotes the set of all elements of $\mathscr{L}$ that are $\vdash$-related to $\Gamma$, i.e. 
\[
C_{\vdash}^{\mathscr{L}}(\Gamma)=\{\alpha\in \mathscr{L}:\Gamma\vdash \alpha\}.
\]
The superscript $\mathscr{L}$ on $C_\vdash$ is dropped when the logical structure under consideration is clear. 

It is easy to see that given a $\vdash$, $C_\vdash$, as described above is an operator from $\mathcal{P}(\mathscr{L})$ to $\mathcal{P}(\mathscr{L})$. Conversely, given an operator $C:\mathcal{P}(\mathscr{L})\to\mathcal{P}(\mathscr{L})$, we can define a relation $\vdash\,\subseteq\mathcal{P}(\mathscr{L})\times\mathscr{L}$ such that $C=C_\vdash$ as follows. For all $\Gamma\cup\{\alpha\}\subseteq\mathscr{L}$, $\Gamma\vdash\alpha$ iff $\alpha\in C(\Gamma)$. We will, on some occasions in this article, define a $\vdash\,\subseteq\mathcal{P}(\mathscr{L})\times\mathscr{L}$ such that $C_\vdash=C$, for some operator $C:\mathcal{P}(\mathscr{L})\to\mathcal{P}(\mathscr{L})$ in this very sense.

The following four definitions are obtained by unraveling \cite[Definition 2]{Beziau2001}. We have chosen to divide the content for the sake of clarity. 

\begin{definition}\label{def:bival}
Let $\mathscr{L}$ be a set. Then any function $v\in \{0,1\}^\mathscr{L}$ is said to be an \emph{$\mathscr{L}$-bivaluation}, or simply a \emph{bivaluation}, when the set under consideration is understood.
\end{definition}

Since $\mathscr{L}$-bivaluations are functions from $\mathscr{L}$ into $\{0,1\}$ (treated just as a set of two distinct elements), these can be understood as characteristic functions of subsets of $\mathscr{L}$ as well. Thus, given any $\Sigma\subseteq\mathscr{L}$, the characteristic function $\chi_\Sigma$ is a bivaluation and conversely, any bivaluation $v$ is the characteristic function of $\{\alpha\in\mathscr{L}\mid v(\alpha)=1\}\subseteq\mathscr{L}$. We will use this alternative description of bivaluations later in the paper.

\begin{definition}\label{def:satisfy}
Suppose $\mathscr{L}$ is a set. An $\mathscr{L}$-bivaluation $v$ is said to \emph{satisfy} $\Gamma\subseteq\mathscr{L}$ if $v(\alpha)=1$ for every $\alpha\in\Gamma$.

Now, suppose $\mathcal{V}$ is a set of $\mathscr{L}$-bivaluations. We say that $\mathcal{V}$ \emph{satisfies} $\Gamma$ if every $v\in\mathcal{V}$ satisfies $\Gamma$.
\end{definition}

Given a set $\mathscr{L}$, a logical structure can also be obtained using a set of $\mathscr{L}$-bivaluations $\mathcal{V}\subseteq \{0,1\}^\mathscr{L}$. The \emph{logical structure induced by} $\mathcal{V}$ is the pair $(\mathscr{L},\vdash_\mathcal{V})$, where $\vdash_{\mathcal{V}}\,\subseteq\mathcal{P}(\mathscr{L})\times\mathscr{L}$ is defined as follows. For any $\Gamma\cup\{\alpha\}\subseteq\mathscr{L}$,
\[
\Gamma\vdash_{\mathcal{V}}\alpha\,\iff\,\hbox{for all }v\in\mathcal{V},\,\hbox{ if }v\hbox{ satisfies }\Gamma\hbox{ then }v(\alpha)=1.
\]

\begin{definition}\label{def:sound_complete}
Let $(\mathscr{L},\vdash)$ be a logical structure and $\mathcal{V}$ a set of bivaluations. We say that $(\mathscr{L},\vdash)$ is \emph{sound} with respect to $(\mathscr{L},\vdash_{\mathcal{V}})$ if for all $\Gamma\cup\{\alpha\}\subseteq\mathscr{L}$, $\Gamma\vdash\alpha\implies \Gamma\vdash_{\mathcal{V}}\alpha$, i.e.\ $\vdash\,\subseteq\,\vdash_\mathcal{V}$.

$(\mathscr{L},\vdash)$ is said to be \emph{complete} with respect to $(\mathscr{L},\vdash_{\mathcal{V}})$ if for all $\Gamma\cup\{\alpha\}\subseteq\mathscr{L}$, $\Gamma\vdash_{\mathcal{V}}\alpha\implies \Gamma\vdash\alpha$, i.e.\ $\vdash_\mathcal{V}\,\subseteq\,\vdash$.

$(\mathscr{L},\vdash)$ is called \emph{adequate} with respect to $(\mathscr{L},\vdash_{\mathcal{V}})$ if it is both sound and complete with respect to $(\mathscr{L},\vdash_{\mathcal{V}})$, i.e.\ if $\vdash\,=\,\vdash_\mathcal{V}$.
\end{definition}

\section{Tarski-type Logical Structures}\label{sec:Tarski}

In this section, we describe the Tarski-type logical structures and prove a characterization theorem for the same. 

As mentioned earlier, these are defined as normal logics in \cite{Beziau2001}. We begin with the definition, which is essentially the same as that of these normal logics.

\begin{definition}\label{def:Tarski-type}
A logical structure $(\mathscr{L},\vdash)$ is said to be of \emph{Tarski-type}, or a \emph{Tarski-type logical structure}, if $\vdash$ satisfies the following properties. 
\begin{enumerate}[label=(\roman*)]
\item For all $\Gamma\subseteq\mathscr{L}$ and $\alpha\in \mathscr{L}$, if $\alpha\in \Gamma$ then $\Gamma\vdash\alpha$. (Reflexivity)
\item For all $\Gamma,\Sigma\subseteq\mathscr{L}$ and $\alpha\in \mathscr{L}$, if $\Gamma\vdash\alpha$ and $\Gamma\subseteq\Sigma$ then $\Sigma\vdash\alpha$. (Monotonicity)
\item For all $\Gamma,\Sigma\subseteq\mathscr{L}$ and $\alpha\in \mathscr{L}$, if $\Gamma\vdash\beta$ for all $\beta\in \Sigma$ and $\Sigma\vdash\alpha$ then $\Gamma\vdash\alpha$. (Transitivity)
\end{enumerate}
If $\vdash$ is reflexive/ monotone/ transitive, the logical structure $(\mathscr{L},\vdash)$ is also said to be reflexive/ monotone/ transitive, respectively.
\end{definition}

\begin{definition}\label{def:closed}
Let $(\mathscr{L},\vdash)$ be a logical structure and $\Gamma\subseteq\mathscr{L}$. $\Gamma$ is said to be \emph{closed} in $\mathscr{L}$ if $C_{\vdash}(\Gamma)=\Gamma$, i.e.\ if for all $\alpha\in \mathscr{L}$, $\Gamma\vdash\alpha$ iff $\alpha\in \Gamma$.
\end{definition}

The above definition of a Tarski-type logical structure can be rephrased in terms of $C_{\vdash}$ as follows. A logical structure $(\mathscr{L},\vdash)$ is of Tarski-type if the following conditions are satisfied. 
\begin{enumerate}[label=(\roman*)]
\item For all $\Gamma\subseteq \mathscr{L}$, $\Gamma\subseteq C_{\vdash}(\Gamma)$. (Reflexivity)
\item For all $\Gamma,\Sigma\subseteq \mathscr{L}$, if $\Gamma\subseteq \Sigma$ then $C_{\vdash}(\Gamma)\subseteq C_{\vdash}(\Sigma)$. (Monotonicity)
\item For all $\Gamma,\Sigma\subseteq \mathscr{L}$, if $\Sigma\subseteq {C_{\vdash}}(\Gamma)$ then ${C_{\vdash}}(\Sigma)\subseteq {C_{\vdash}}(\Gamma)$. (Transitivity)
\end{enumerate}
It may be noted here that a similar definition of a Tarskian consequence operator has been given in \cite{Tsuji1998}. We now adapt a couple of results connecting the Tarskian consequence operators with the concept of a \emph{Suszko set} of bivaluations from \cite{Tsuji1998}.

\begin{definition}\label{def:Suszko}\cite[Subsection 3.1]{Tsuji1998}
Let $(\mathscr{L},\vdash)$ be a logical structure. A nonempty set of bivaluations $\mathcal{V}$ is called a \emph{Suszko set for $(\mathscr{L},\vdash)$}, or simply a \emph{Suszko set} (when the logical structure is clear from the context), if $(\mathscr{L},\vdash)$ is adequate with respect to $(\mathscr{L},\vdash_{\mathcal{V}})$. In this case, we will also call $(\mathscr{L},\vdash)$ to be \emph{bivalent}.
\end{definition}

Now, by \cite[Corollary 1]{Tsuji1998}, if $(\mathscr{L},\vdash)$ is a Tarski-type logical structure, then there exists a Suszko set for $(\mathscr{L},\vdash)$.

Conversely, by \cite[Corollary 2]{Tsuji1998}, if there exists a Suszko set for a logical structure $(\mathscr{L},\vdash)$, then $\vdash$ is reflexive. Moreover, it is, in fact, straightforward to verify that, given any set $\mathscr{L}$ and a nonempty $\mathcal{V}\subseteq\{0,1\}^\mathscr{L}$, $(\mathscr{L},\vdash_{\mathcal{V}})$ is of Tarski-type. So in particular, if there exists a Suszko set $\mathcal{V}$ for a logical structure $(\mathscr{L},\vdash)$, then $\vdash\,=\,\vdash_\mathcal{V}$, and so $(\mathscr{L},\vdash)$ must be of Tarski-type. Thus we get the following representation theorem for Tarski-type logical structures.

\begin{theorem}[\textsc{Representation Theorem for Tarski-type}]\label{thm:Rep_Tarski}
Suppose $(\mathscr{L},\vdash)$ is a logical structure. Then $(\mathscr{L},\vdash)$ is of Tarski-type iff there exists a Suszko set for $(\mathscr{L},\vdash)$.
\end{theorem}

\begin{remark}
A more general version of the above theorem, although in a many-valued context, is available in \cite{Chakraborty_Dutta2019}.
\end{remark}

We now come to our main result in this section, the characterization of Tarski-type logical structures. We begin with the definition of strongly closed sets in a given logical structure.

\begin{definition}\label{def:str_closed}
Let $(\mathscr{L},\vdash)$ be a logical structure and $\Gamma\subseteq\mathscr{L}$. Then $\Gamma$ is called \emph{strongly closed} in $\mathscr{L}$ if it satisfies the following conditions.
\begin{enumerate}[label=(\roman*)]
\item For all $\alpha\in\Gamma$, $\Gamma\vdash\alpha$, i.e.\ $\Gamma\subseteq C_\vdash(\Gamma)$.
\item For all $\alpha\in\mathscr{L}$, if there exists a $\Gamma^\prime\subseteq\Gamma$ with $\Gamma^\prime\vdash\alpha$ then $\alpha\in\Gamma$.
\end{enumerate}
\end{definition}

\begin{remark}\label{rem:closed-str_closed}
It is easy to see that every strongly closed set is closed, but the converse is not true for arbitrary logical structures. However, if the logical structure under consideration is monotone, then the two concepts are equivalent.
\end{remark}

\begin{theorem}[\textsc{Characterization of Tarski-type}]\label{thm:Char_Tar}
Let $(\mathscr{L},\vdash)$ be a logical structure. Then the following statements are equivalent.

\begin{enumerate}[label=(\arabic*)]
\item $(\mathscr{L},\vdash)$ is of Tarski-type.
\item There exists a Suszko set for $(\mathscr{L},\vdash)$.
\item For all $\Gamma\cup\{\alpha\}\subseteq\mathscr{L}$ with $\Gamma\nvdash\alpha$, there exists a strongly closed extension $\Sigma\supseteq\Gamma$ such that $\Sigma\nvdash\alpha$.
\item For all $\Gamma,\Sigma\subseteq \mathscr{L}$, $\Sigma\subseteq{C_{\vdash}}(\Gamma)$ iff ${C_{\vdash}}(\Sigma)\subseteq {C_{\vdash}}(\Gamma)$.
\end{enumerate}
\end{theorem}

\begin{proof}
The equivalence of (1) and (2) constituted Theorem \ref{thm:Rep_Tarski}. Thus to complete the proof, it is enough to show that statements (3) and (4) are equivalent to (1). 

\underline{(1)$\implies$(3)}: Suppose $\Gamma\cup\{\alpha\}\subseteq\mathscr{L}$ such that $\Gamma\nvdash\alpha$. We will show that $C_{\vdash}(\Gamma)$ is a strongly closed extension of $\Gamma$ such that $C_\vdash(\Gamma)\nvdash\alpha$. 

Since $\vdash$ is reflexive, $\Gamma\subseteq C_\vdash(\Gamma)$. Then by monotonicity of $\vdash$, we have $C_\vdash(\Gamma)\subseteq C_\vdash(C_\vdash(\Gamma))$. Moreover, by transitivity of $\vdash$, it follows that, $C_\vdash(C_\vdash(\Gamma))\subseteq C_\vdash(\Gamma)$. Thus $C_\vdash(C_\vdash(\Gamma))=C_\vdash(\Gamma)$, and hence $C_\vdash(\Gamma)$ is closed. Since $\vdash$ is monotone, by Remark \ref{rem:closed-str_closed}, we can conclude that $C_\vdash(\Gamma)$ is also strongly closed.

Finally, to prove that $C_\vdash(\Gamma)\nvdash\alpha$, we note that since $\Gamma\nvdash\alpha$, $\alpha\notin C_\vdash(\Gamma)$. This implies that $C_\vdash(\Gamma)\nvdash\alpha$ as $C_\vdash(\Gamma)$ is closed. Thus statement (3) holds.

\underline{(3)$\implies$(1)}: Let $\Gamma\subseteq \mathscr{L}$ and $\beta\in\Gamma$. We claim that $\Gamma\vdash\beta$. Suppose the contrary, i.e.\ $\Gamma\nvdash\beta$. Then there exists a strongly closed $\Sigma\supseteq\Gamma$ with $\Sigma\nvdash\beta$. Now, since $\Sigma$ is strongly closed, and hence closed,  $\Sigma\nvdash\beta$ implies that $\beta\notin\Sigma$. Thus $\beta\notin \Gamma$ as well, contrary to our hypothesis. Hence $\Gamma\vdash\beta$. Thus $\vdash$ is reflexive.

Next, let $\Gamma\subseteq\Sigma\subseteq\mathscr{L}$ and $\beta\in\mathscr{L}$ such that $\Gamma\vdash\beta$. We claim that $\Sigma\vdash\beta$. Suppose the contrary, i.e.\ $\Sigma\nvdash\beta$. Then there exists a strongly closed $\Delta\supseteq\Sigma$ with $\Delta\nvdash\beta$. Now, $\Gamma\subseteq\Delta$ with $\Gamma\vdash\beta$. So, $\Delta\vdash\beta$ since $\Delta$ is strongly closed. This contradiction proves that $\Sigma\vdash\beta$. Thus $\vdash$ is monotone.

Finally, let $\Gamma,\Sigma\subseteq \mathscr{L}$ and $\beta\in\mathscr{L}$ such that $\Gamma\vdash\gamma$ for all $\gamma\in \Sigma$, and $\Sigma\vdash\beta$. We claim that $\Gamma\vdash\beta$. Suppose the contrary, i.e.\ $\Gamma\nvdash\beta$. Then there exists a strongly closed set $\Delta\supseteq\Gamma$ with $\Delta\nvdash\beta$. Now, since $\Gamma\vdash\gamma$ for all $\gamma\in \Sigma$, $\Gamma\subseteq \Delta$ and $\Delta$ is strongly closed, it follows that $\gamma\in\Delta$ for all $\gamma\in\Sigma$. Thus $\Sigma\subseteq\Delta$. Now, since $\Sigma\vdash\beta$, we must have $\Delta\vdash\beta$, since $\Delta$ is strongly closed. This is a contradiction. Hence $\Gamma\vdash\beta$. So $\vdash$ is transitive.

Thus $(\mathscr{L},\vdash)$ is a Tarski-type logical structure.

We now show that statements (1) and (4) are equivalent using the definition of a Tarski-type logical structure in terms of $C_\vdash$.

\underline{(1)$\implies$(4)}: Let $\Gamma,\Sigma\subseteq\mathscr{L}$. Now, if $\Sigma\subseteq C_\vdash(\Gamma)$, then $C_\vdash(\Sigma)\subseteq C_\vdash(\Gamma)$ by transitivity. Conversely, suppose ${C_{\vdash}}(\Sigma)\subseteq {C_{\vdash}}(\Gamma)$. Then by reflexivity, we get $\Sigma \subseteq {C_{\vdash}}(\Sigma)\subseteq {C_{\vdash}}(\Gamma)$. Thus $\Sigma\subseteq C_\vdash(\Gamma)$ iff ${C_{\vdash}}(\Sigma)\subseteq {C_{\vdash}}(\Gamma)$. Hence statement (4) holds.

\underline{(4)$\implies$(1)}: We first note that reflexivity and transitivity follow immediately from the assumed condition. To prove monotonicity, let $\Sigma\subseteq\Gamma\subseteq\mathscr{L}$. By reflexivity, we have $\Gamma\subseteq C_{\vdash}(\Gamma)$. Thus $\Sigma\subseteq C_{\vdash}(\Gamma)$. Then by the assumed condition, it follows that $C_{\vdash}(\Sigma)\subseteq C_{\vdash}(\Gamma)$. Thus $(\mathscr{L},\vdash)$ is a Tarski-type logical structure.
\end{proof}

\begin{remark}
The statement (3) in the above theorem can be tightened to show that $(\mathscr{L},\vdash)$ is of Tarski-type iff for all $\Gamma\cup\{\alpha\}\subseteq\mathscr{L}$ with $\Gamma\nvdash\alpha$, there exists a $\subseteq$-minimum strongly closed extension $\Delta\supseteq\Gamma$ such that $\Delta\nvdash\alpha$. The forward direction follows by choosing $\Delta=C_{\vdash}(\Gamma)$ as in the proof of (1)$\implies$(3) above. The converse is immediate.
\end{remark}

\section{Saturated Sets and Maximal Nontrivial Sets}\label{sec:Sat_MaxNontriv}

We devote this section to saturated and maximal nontrivial sets, which play crucial roles in the characterizations of Lindenbaum-type logical structures introduced in Section \ref{sec:Lind}. As mentioned earlier, the concept of $\alpha$-saturated sets has been described in \cite{Loparic_daCosta1984}, and in \cite{Beziau1994} as $a$-excessive theories.

\subsection{Saturated sets}\label{subsec:SatSets}

\begin{definition}\label{def:alpha-sat} 
Let $(\mathscr{L},\vdash)$ be a logical structure and $\Gamma\cup\{\alpha\}\subseteq\mathscr{L}$. $\Gamma$ is called \emph{$\alpha$-saturated} in $(\mathscr{L},\vdash)$ if $\Gamma\nvdash\alpha$ but for any $\beta\in \mathscr{L}\setminus\Gamma$, $\Gamma\cup\{\beta\}\vdash\alpha$.

$\Gamma$ is called \emph{saturated} in $(\mathscr{L},\vdash)$ if it is $\alpha$-saturated for some $\alpha\in \mathscr{L}$.
\end{definition}

\begin{definition}\label{def:relmax}\cite[Definition 7]{Beziau2001}
Let $(\mathscr{L},\vdash)$ be a logical structure and $\Gamma\cup\{\alpha\}\subseteq\mathscr{L}$. $\Gamma$ is said to be \emph{relatively maximal in $\alpha$} in $(\mathscr{L},\vdash)$ if $\Gamma\nvdash\alpha$ but for any $\Sigma\supsetneq\Gamma$, $\Sigma\vdash\alpha$.
\end{definition}

\begin{remark}\label{rem:alpha-sat}
Suppose $(\mathscr{L},\vdash)$ is a logical structure and $\Gamma\cup\{\alpha\}\subseteq\mathscr{L}$. Then the following observations are easy to check.
\begin{enumerate}[label=(\roman*)]
    \item If $\Gamma$ is relatively maximal in $\alpha$, then it is $\alpha$-saturated.
    \item If $(\mathscr{L},\vdash)$ is monotone, then every $\alpha$-saturated set is relatively maximal in $\alpha$.
\end{enumerate}
\end{remark}

\begin{example}\label{exm:cpc_alpha-sat}
Let $(\mathscr{L}_\mathbf{CL},\vdash_{\mathbf{CL}})$ denote the classical propositional logic. Then it is a straightforward observation that every maximal consistent set $\Sigma$ of $(\mathscr{L}_\mathbf{CL},\vdash_{\mathbf{CL}})$ is $\alpha$-saturated, and also relatively maximal in $\alpha$, in the underlying logical structure $(\lvert\mathscr{L}_\mathbf{CL}\rvert,\vdash_{\mathbf{CL}})$ for each $\alpha\notin\Sigma$.
\end{example} 

\begin{definition}{\label{def:cut,mcut}}
A logical structure $(\mathscr{L},\vdash)$ is said to satisfy \textit{cut} if for all $\Gamma\cup\{\alpha,\beta\}\subseteq \mathscr{L}$, $\Gamma\vdash\alpha$ and $\Gamma\cup\{\alpha\}\vdash\beta$ implies that $\Gamma\vdash\beta$.

$(\mathscr{L},\vdash)$ is said to satisfy \textit{mixed-cut} if for all $\Gamma\cup\Sigma\cup\{\alpha,\beta\}\subseteq \mathscr{L}$, $\Gamma\vdash\alpha$ and $\Sigma\cup\{\alpha\}\vdash\beta$ implies that $\Gamma\cup\Sigma\vdash\beta$.
\end{definition} 

\begin{remark}
It is straightforward to see that cut is a special case of mixed-cut, i.e. if a logical structure satisfies mixed-cut, then it also satisfies cut. 

The above notions of cut and mixed-cut have been adopted from those in \cite{Hlobil2018}.
\end{remark}

\begin{theorem}{\label{thm:cut=>closed_sat}}
Let $(\mathscr{L},\vdash)$ be a logical structure satisfying cut. Then every saturated set is closed.
\end{theorem}

\begin{proof}
Let $\Sigma\subseteq\mathscr{L}$ be a saturated set. So, $\Sigma$ is $\beta$-saturated for some $\beta\in\mathscr{L}$. If possible, suppose $\Sigma$ is not closed. Then there exists $\alpha\in\mathscr{L}$ such that $\Sigma\vdash\alpha$ but $\alpha\notin \Sigma$. Now, since $\Sigma$ is $\beta$-saturated and $\alpha\notin \Sigma$, we have $\Sigma\cup\{\alpha\}\vdash\beta$. Thus $\Sigma\vdash \alpha$ and $\Sigma\cup\{\alpha\}\vdash \beta$. Then by cut, it follows that $\Sigma\vdash \beta$. This is a contradiction as $\Sigma$ is $\beta$-saturated. Hence $\Sigma$ must be closed.
\end{proof}

\begin{corollary}{\label{cor:cut=>pRL}}
Let $(\mathscr{L},\vdash)$ be a logical structure that satisfies cut. Then $C_\vdash(\mathscr{L})=\mathscr{L}$.
\end{corollary}

\begin{proof}
Suppose $C_\vdash(\mathscr{L})\neq\mathscr{L}$. Then there exists an $\alpha\in\mathscr{L}$ such that $\mathscr{L}\nvdash\alpha$. This implies that $\mathscr{L}$ is saturated. Then as $(\mathscr{L},\vdash)$ satisfies cut, it follows by the above theorem that $\mathscr{L}$ is closed. However, this implies that $\mathscr{L}\vdash\alpha$, since $\alpha\in\mathscr{L}$. This is a contradiction. Hence $C_\vdash(\mathscr{L})=\mathscr{L}$.
\end{proof}

\begin{theorem}{\label{thm:mcut=>str-closed_sat}}
Let $(\mathscr{L},\vdash)$ be a logical structure satisfying mixed-cut. Then every saturated set is strongly closed.
\end{theorem}

\begin{proof}
Let $\Sigma\subseteq\mathscr{L}$ be a saturated set. Now, since $(\mathscr{L},\vdash)$ satisfies mixed-cut, it also satisfies cut. Then by Theorem \ref{thm:cut=>closed_sat}, $\Sigma$ is closed. Thus, for any $\alpha\in\mathscr{L}$, $\Sigma\vdash\alpha$ iff $\alpha\in\Sigma$. Hence to show that $\Sigma$ is strongly closed, it suffices to show that, for any $\alpha\in\mathscr{L}$, if $\Gamma\subsetneq\Sigma$ with $\Gamma\vdash\alpha$, then $\alpha\in\Sigma$.

Suppose the contrary, i.e. there exists $\alpha\in\mathscr{L}$ such that $\Gamma\subsetneq\Sigma$ and $\Gamma\vdash\alpha$ but $\alpha\notin \Sigma$. Now, $\Sigma$ is $\beta$-saturated for some $\beta\in\mathscr{L}$. Thus $\Sigma\nvdash\beta$ but since $\alpha\notin\Sigma$, $\Sigma\cup\{\alpha\}\vdash\beta$. So, we have $\Gamma\vdash\alpha$ and $\Sigma\cup\{\alpha\}\vdash\beta$, and hence by mixed-cut, $\Gamma\cup\Sigma\vdash\beta$. This implies that $\Sigma\vdash\beta$ since $\Gamma\subsetneq\Sigma$. This is a contradiction. Hence $\Sigma$ must be strongly closed.
\end{proof}

\subsection{Maximal nontrivial sets}\label{subsec:MaxNontrivSets}

\begin{definition}\label{def:trivial}
Let $(\mathscr{L},\vdash)$ be a logical structure and $\Gamma\subseteq\mathscr{L}$. $\Gamma$ is called \emph{trivial} if $C_{\vdash}(\Gamma)=\mathscr{L}$, and \emph{nontrivial} otherwise. Thus, $\Gamma\subseteq\mathscr{L}$ is nontrivial if there exists an $\alpha\in\mathscr{L}$ such that $\Gamma\nvdash\alpha$.
\end{definition}

\begin{remark}
A trivial set is also called an \emph{absolutely inconsistent} set.
\end{remark}

\begin{definition}\label{def:maxntriv}
Let $(\mathscr{L},\vdash)$ be a logical structure and $\Gamma\subseteq\mathscr{L}$. $\Gamma$ is said to be \emph{maximal nontrivial} in $(\mathscr{L},\vdash)$ if $\Gamma$ is nontrivial but every $\Sigma\subseteq\mathscr{L}$ with $\Gamma\subsetneq\Sigma$ is trivial.
\end{definition}

\begin{theorem}\label{thm:maxnontriv=>relmax}
Let $(\mathscr{L},\vdash)$ be a logical structure and $\Gamma\subseteq \mathscr{L}$. If $\Gamma$ is maximal nontrivial then it is relatively maximal in $\alpha$ for every $\alpha\notin C_\vdash(\Gamma)$.
\end{theorem}

\begin{proof}
Suppose $\Gamma\subseteq\mathscr{L}$ is maximal nontrivial. Let $\alpha\notin C_\vdash(\Gamma)$ (such an $\alpha$ exists since $\Gamma$ is nontrivial, and hence $C_\vdash(\Gamma)\neq\mathscr{L}$). Then $\Gamma\nvdash\alpha$. Let $\Sigma\subseteq\mathscr{L}$ such that $\Gamma\subsetneq\Sigma$. Then as $\Gamma$ is maximal nontrivial, $\Sigma$ must be trivial. So, in particular, $\Sigma\vdash\alpha$. Hence $\Gamma$ is relatively maximal in $\alpha$ for every $\alpha\notin C_\vdash(\Gamma)$.
\end{proof}

\begin{corollary}\label{cor:maxnontriv=>sat}
Let $(\mathscr{L},\vdash)$ be a logical structure and $\Gamma\subseteq \mathscr{L}$. If $\Gamma$ is maximal nontrivial, then it is $\alpha$-saturated for every $\alpha\notin C_\vdash(\Gamma)$, and hence saturated. 
\end{corollary}

\begin{proof}
Straightforward from the above theorem and Remark \ref{rem:alpha-sat}.
\end{proof}

\begin{remark}\label{rem:sat=>nontriv}
Suppose $(\mathscr{L},\vdash)$ is a logical structure and $\Gamma\subseteq\mathscr{L}$ is saturated. Then there exists an $\alpha\in\mathscr{L}$ such that $\Gamma$ is $\alpha$-saturated. This implies that $\Gamma\nvdash\alpha$. Hence $\Gamma$ is nontrivial. Thus every saturated ($\alpha$-saturated) set is nontrivial.
\end{remark}

The converse of Theorem \ref{thm:maxnontriv=>relmax} also holds for monotone logical structures, as shown below.

\begin{theorem}
Let $(\mathscr{L},\vdash)$ be a monotone logical structure and $\Gamma\subseteq \mathscr{L}$. If $\Gamma$ relatively maximal in $\alpha$ for every $\alpha\notin C_\vdash(\Gamma)$ then $\Gamma$ is maximal nontrivial.
\end{theorem}

\begin{proof}
Suppose $\Gamma\subseteq\mathscr{L}$ is relatively maximal in $\alpha$ for every $\alpha\notin C_\vdash(\Gamma)$ and $\Sigma\supsetneq\Gamma$. If possible, let $\Sigma$ be nontrivial, i.e. there exists $\beta\in\mathscr{L}$ such that $\Sigma\nvdash\beta$. This implies that $\Gamma\nvdash\beta$ as well, since $(\mathscr{L},\vdash)$ is monotone and $\Gamma\subsetneq \Sigma$. So, $\beta\notin C_\vdash(\Gamma)$, and hence by our assumption, $\Gamma$ is relatively maximal in $\beta$. Then it must be that $\Sigma\vdash\beta$. This is a contradiction. So, $\Sigma$ cannot be nontrivial, and thus $\Gamma$ is maximal nontrivial.
\end{proof}

We now digress a bit and show that a certain generalization of \emph{implication-saturated} sets, introduced by Batens in \cite{Batens1980}, constitutes an example of maximal nontrivial sets, which in turn are saturated sets as well, by Corollary \ref{cor:maxnontriv=>sat}. In \cite{Batens1980}, the implication-saturated sets have been used to develop an alternative method for proving the completeness results for a certain class of logics. It has been argued there that since this method avoids the use of traditional maximal consistent sets, it can be applied to a wide class of non-classical logics, especially the paraconsistent logics described in \cite{Arruda1980, daCosta1974, daCosta_Alves1977}.

\begin{definition}\label{def:imp-sat}
Let $(\mathscr{L},\vdash)$ be a logic with a binary connective, $\limp$, and $\Gamma\subseteq\lvert\mathscr{L}\rvert$. $\Gamma$ is called \emph{$\limp$-saturated} if for every $\alpha\notin\Gamma$, $\xi_\alpha:=\{\alpha\limp\beta:\beta\in \mathscr{L}\}\subseteq\Gamma$.
\end{definition}

\begin{remark} The above definition, although motivated by the notion of implication-saturated sets, is a slight generalization of the latter since $\limp$ here could be any binary connective.

It may also be noted here that the discussion in \cite{Batens1980} is based on logics and not logical structures. However, as mentioned in Remark \ref{rem:logicVlog_str}, every logic has an underlying logical structure. Hence a connection can be drawn regardless of this difference.

The concept of \emph{deductively closed} sets used in \cite{Batens1980} is the same as that of our closed sets described in Definition \ref{def:closed}.

We differ in the definition of trivial sets as well. While in \cite{Batens1980}, a set $\Gamma$ in a logic $(\mathscr{L},\vdash)$ is called trivial if $\Gamma=\mathscr{L}$ ($\Gamma=\lvert\mathscr{L}\rvert$ in our notation), we call a set $\Gamma$ in a logical structure $(\mathscr{L},\vdash)$ trivial if $C_{\vdash}(\Gamma)=\mathscr{L}$. However, this difference is inconsequential if we are dealing with closed sets. 
\end{remark}

Suppose $(\mathscr{L},\vdash)$ is a logic with $\limp$ such that, for every $\Gamma\cup\{\alpha,\beta\}\subseteq\lvert\mathscr{L}\rvert$, $\Gamma\vdash\alpha$ and $\Gamma\vdash\alpha\limp\beta$ imply $\Gamma\vdash\beta$. Then we say that modus ponens holds in $(\mathscr{L},\vdash)$.

\begin{theorem}\label{thm:imp-sat=>maxntriv}
Suppose $(\mathscr{L},\vdash)$ is a reflexive logic with $\limp$ such that modus ponens holds in $(\mathscr{L},\vdash)$. Then every nontrivial $\limp$-saturated $\Gamma\subseteq\lvert\mathscr{L}\rvert$ is maximal nontrivial and closed.
\end{theorem}

\begin{proof}
Let $\Gamma\subseteq\lvert\mathscr{L}\rvert$ be nontrivial and $\limp$-saturated. Suppose $\Gamma$ is not maximal nontrivial. Then there exists $\Sigma\supsetneq \Gamma$ such that $\Sigma$ is nontrivial. Since $\Gamma\subsetneq \Sigma$, there exists $\gamma\in \Sigma\setminus \Gamma$. Then $\xi_\gamma\subseteq \Gamma$, as $\Gamma$ is $\limp$-saturated, and hence $\xi_\gamma\subseteq\Sigma$. Thus $\gamma\in\Sigma$ and $\gamma\limp\alpha\in \Sigma$ for every $\alpha\in \lvert\mathscr{L}\rvert$. Therefore by reflexivity of $(\mathscr{L},
\vdash)$, we have $\Sigma\vdash\gamma$ and $\Sigma\vdash\gamma\limp\alpha$ for every $\alpha\in\lvert\mathscr{L}\rvert$. Then by modus ponens, we get $\Sigma\vdash\alpha$ for every $\alpha\in\lvert\mathscr{L}\rvert$, i.e. $\Sigma$ is trivial, a contradiction. Since $\Gamma$ is nontrivial, this implies that $\Gamma$ is maximal nontrivial.

Next, to show that $\Gamma$ is closed, suppose the contrary. Then there exists a $\beta\in\lvert\mathscr{L}\rvert$ such that $\Gamma\vdash\beta$ but $\beta\notin \Gamma$. Since $\Gamma$ is implication saturated, this implies that $\xi_\beta\subseteq \Gamma$, i.e. $\beta\limp\delta\in \Gamma$ for all $\delta\in\lvert\mathscr{L}\rvert$. So, by reflexivity, $\Gamma\vdash\beta\limp\delta$ for all $\delta\in\lvert\mathscr{L}\rvert$. Hence by modus ponens, $\Gamma\vdash\delta$ for all $\delta\in\lvert\mathscr{L}\rvert$. In other words, $\Gamma$ is trivial, a contradiction. Thus $\Gamma$ must be closed. 
\end{proof}

\begin{theorem}\label{thm:imp-sat}
Let $(\mathscr{L},\vdash)$ be a reflexive logic with $\limp$. Suppose that for all $\Sigma\cup\{\beta\}\subseteq\lvert\mathscr{L}\rvert$ with $\beta\notin\Sigma$, $C_{\vdash}(\Sigma\cup\xi_\beta)\subseteq C_{\vdash}(\Sigma)$. Then any closed $\Gamma\subseteq\lvert\mathscr{L}\rvert$ is $\limp$-saturated.
\end{theorem}

\begin{proof}
Suppose $\Gamma\cup\{\alpha\}\subseteq\lvert\mathscr{L}\rvert$ such that $\Gamma$ is closed and $\alpha\notin\Gamma$. Now, 
\[
\begin{array}{lcll}
\xi_\alpha&\subseteq&\Gamma\cup\xi_\alpha\\
&\subseteq&C_{\vdash}(\Gamma\cup\xi_\alpha)&\hbox{(by reflexivity)}\\
&\subseteq&C_{\vdash}(\Gamma)&\hbox{(by the given condition)}\\
&=&\Gamma&\hbox{(since $\Gamma$ is closed)}
\end{array}
\]
Thus $\Gamma$ is $\limp$-saturated.
\end{proof}

\begin{remark}
The proviso that for all $\Sigma\cup\{\beta\}\subseteq\lvert\mathscr{L}\rvert$ with $\beta\notin\Sigma$, $C_{\vdash}(\Sigma\cup\xi_\beta)\subseteq C_{\vdash}(\Sigma)$, in Theorem \ref{thm:imp-sat} is what \cite[Lemma 4]{Batens1980} states.
\end{remark}

We can now conclude that in a logic that satisfies the conditions of Theorems \ref{thm:imp-sat=>maxntriv} and \ref{thm:imp-sat}, a nontrivial set is $\limp$-saturated iff it is closed and maximal nontrivial in the corresponding underlying logical structure.

\begin{remark}
It is noteworthy that while the notions of saturated and maximal nontrivial sets do not depend on the existence of any connective like $\limp$, and hence make sense for any logical structure, $\limp$-saturated sets can only exist in logics with such a logical operator. The above results thus provide a more general alternative understanding of the discussions in \cite{Batens1980}.
\end{remark}

\section{Lindenbaum-type Logical Structures}\label{sec:Lind}

In \cite[Section 4]{Beziau1999}, four versions of Lindenbaum's law were listed. We now list the definitions of the logical structures induced by them, albeit with some terminological modifications.

\begin{definition}{\label{def:Lind-type}}
A logical structure $(\mathscr{L},\vdash)$ is said to be of:

\begin{enumerate}[label=(\alph*)]
    \item \emph{Lindenbaum-I-type} if for all nontrivial $\Gamma\subseteq \mathscr{L}$, there exists a saturated $\Sigma\supseteq \Gamma$.\label{def:LindI} 
     \item \emph{Lindenbaum-II-type} if for all $\Gamma\cup\{\alpha\}\subseteq \mathscr{L}$, with $\Gamma\nvdash\alpha$ there exists an $\alpha$-saturated $\Sigma\supseteq \Gamma$.\label{def:LindII}
     \item \emph{Lindenbaum-III-type} if for all nontrivial $\Gamma\subseteq \mathscr{L}$, there exists a maximal nontrivial $\Sigma\supseteq \Gamma$.\label{def:LindIII}
     \item \emph{Lindenbaum-IV-type} if for all $\Gamma\cup\{\alpha\}\subseteq \mathscr{L}$, with $\Gamma\nvdash\alpha$ there exists a $\Sigma\supseteq \Gamma$ which is relatively maximal in $\alpha$.\label{def:LindIV}
\end{enumerate}
\end{definition}

\begin{remark}
There are important terminological differences between the above definitions and the translations from French of those in \cite{Beziau1999}. These are summarized below.

\begin{center}
\begin{tabular}{|l|l|}
\hline
\textbf{In \cite{Beziau1999}}&\textbf{In this paper}\\
\hline
$(\mathscr{L},\vdash)$ is a logic&$(\mathscr{L},\vdash)$ is a logical structure\\ 
\hline
$\Gamma\subseteq\mathscr{L}$ is a limited theory&$\Gamma$ is nontrivial\\
\hline
$\Gamma$ is $\alpha$-limited& $\Gamma\nvdash\alpha$\\ 
\hline
$\Gamma$ is excessive&$\Gamma$ is saturated\\ 
\hline
$\Gamma$ is $\alpha$-excessive&$\Gamma$ is $\alpha$-saturated\\ \hline
$\Gamma$ is maximal&$\Gamma$ is maximal nontrivial\\ 
\hline
$\Gamma$ is $\alpha$-maximal&$\Gamma$ is relatively maximal in $\alpha$\\ 
\hline
Weak Lindenbaum-Asser logic&Lindenbaum-I-type logical structure\\ \hline
Lindenbaum-Asser logic&Lindenbaum-II-type logical structure\\ 
\hline
Lindenbaum logic&Lindenbaum-III-type logical structure\\ 
\hline
Strong Lindenbaum logic&Lindenbaum-IV-type logical structure\\ 
\hline
\end{tabular}
\end{center}
\end{remark}
 
\subsection{Lindenbaum-II-type logical structures}\label{subsec:LindII}

\begin{theorem}\label{thm:LindII}
Let $(\mathscr{L},\vdash)$ be a Lindenbaum-II-type logical structure and $\alpha\in \mathscr{L}$. Then every maximal $\alpha$-saturated set is relatively maximal in $\alpha$.
\end{theorem}

\begin{proof}
Let $\Gamma\cup\{\alpha\}\subseteq\mathscr{L}$ such that $\Gamma$ is maximal $\alpha$-saturated. If possible, suppose $\Gamma$ is not relatively maximal in $\alpha$. Then there exists $\Sigma\supsetneq\Gamma$ such that $\Sigma\nvdash\alpha$. Since $(\mathscr{L},\vdash)$ is of Lindenbaum-II-type, this implies that there must exist an $\alpha$-saturated $\Delta\supseteq \Sigma$ by Definition \ref{def:Lind-type}\ref{def:LindII}. Thus $\Delta$ is an $\alpha$-saturated set containing $\Gamma$. This contradicts the maximality of $\Gamma$. Hence $\Gamma$ must be relatively maximal in $\alpha$.
\end{proof}

\begin{remark} 
The following result (translated from French and rephrased using our terminology) is quoted from \cite[Section 5]{Beziau1999}.

\begin{quote}
``In a Lindenbaum-II-type logical structure, every set coincides with the intersection of its saturated extensions.''
\end{quote}

We point out here that the above result is not true in general. This is witnessed in the following example. Let $(\mathscr{L}_\mathbf{CL},\vdash_{\mathbf{CL}})$ be as before. We define a new logical structure $(\mathscr{L},\vdash)$, where $\mathscr{L}=\lvert\mathscr{L}_\mathbf{CL}\rvert$, and for all $\Gamma\cup\{\alpha\}\subseteq\mathscr{L}$, $\Gamma\vdash\alpha$ iff $\Gamma\nvdash_\mathbf{CL}\alpha$. It then immediately follows that $\mathscr{L}\nvdash\alpha$ for all $\alpha\in\mathscr{L}$, and hence $\mathscr{L}$ is trivially $\alpha$-saturated in $(\mathscr{L},\vdash)$ for every $\alpha\in\mathscr{L}$. So, for every $\Gamma\cup\{\alpha\}$ with $\Gamma\nvdash\alpha$, there exists an $\alpha$-saturated $\Sigma\supseteq\Gamma$. Thus $(\mathscr{L},\vdash)$ is a Lindenbaum-II-type logical structure.

Now, suppose $\Gamma\subseteq\mathscr{L}$ is $\alpha$-saturated for some $\alpha\in\mathscr{L}$. Then $\Gamma\nvdash\alpha$ which implies that $\Gamma\vdash_\mathbf{CL}\alpha$. This in turn implies that $\Gamma\cup\{\delta\}\vdash_\mathbf{CL}\alpha$, and hence $\Gamma\cup\{\delta\}\nvdash\alpha$, for any $\delta\in\mathscr{L}\setminus\Gamma$. This is impossible unless $\Gamma=\mathscr{L}$. Thus $\mathscr{L}$ is the only saturated set in $(\mathscr{L},\vdash)$. Hence in this case, for any $\Gamma\subseteq\mathscr{L}$, the intersection of its saturated extensions is $\mathscr{L}$, and so if $\Gamma\subsetneq\mathscr{L}$, it cannot be equal to the intersection of its saturated extensions.
\end{remark}

However, the following weaker versions of the above result hold.

\begin{theorem}
Let $(\mathscr{L},\vdash)$ be a Lindenbaum-II-type logical structure satisfying cut. If $\Gamma\subsetneq\mathscr{L}$ is closed, then it coincides with the intersection of its saturated extensions.
\end{theorem}

\begin{proof}
Let $\Gamma\cup\{\alpha\}\subseteq\mathscr{L}$ such that $\Gamma\subsetneq\mathscr{L}$ is closed and $\alpha\notin\Gamma$. Since $\Gamma$ is closed, we have $\Gamma\nvdash\alpha$. Then, as $(\mathscr{L},\vdash)$ is of Lindenbaum-II-type, by Definition \ref{def:Lind-type}\ref{def:LindII}, there exists an $\alpha$-saturated $\Delta\supseteq \Gamma$. This implies that $\Delta\nvdash\alpha$, i.e. $\alpha\notin C_{\vdash}(\Delta)$. Now, as $\Delta$ is saturated and $(\mathscr{L},\vdash)$ satisfies cut, by Theorem~\ref{thm:cut=>closed_sat}, $\Delta$ is closed. Hence $C_{\vdash}(\Delta)=\Delta$ and so $\alpha\notin\Delta$. Let $\mathbb{SAT}$ denote the collection of all saturated sets in $(\mathscr{L},\vdash)$. Then this implies that $\alpha\notin\displaystyle\bigcap_{\substack{\Gamma\subseteq\Sigma\\\Sigma\in \mathbb{SAT}}}\Sigma$. Since $\alpha$ was chosen arbitrarily, we have $\Gamma\supseteq\displaystyle\bigcap_{\substack{\Gamma\subseteq\Sigma\\\Sigma\in \mathbb{SAT}}}\Sigma$. The reverse inclusion is immediate. Hence $\Gamma=\displaystyle\bigcap_{\substack{\Gamma\subseteq\Sigma\\\Sigma\in \mathbb{SAT}}}\Sigma$.
\end{proof}

\begin{theorem}
Let $(\mathscr{L},\vdash)$ be a reflexive Lindenbaum-II-type logical structure. If $\Gamma\subsetneq\mathscr{L}$ is closed, then $\Gamma=\displaystyle\bigcap_{\substack{\Gamma\subseteq\Sigma\\\Sigma\in \mathbb{SAT}}}C_\vdash(\Sigma)$, where $\mathbb{SAT}$ denotes, as before, the set of saturated subsets of $\mathscr{L}$.
\end{theorem}

\begin{proof}
Let $\Gamma\cup\{\alpha\}\subseteq\mathscr{L}$ such that $\Gamma\subsetneq\mathscr{L}$ is closed and $\alpha\notin\Gamma$. Since $\Gamma$ is closed, we have $\Gamma\nvdash\alpha$. Then, as $(\mathscr{L},\vdash)$ is of Lindenbaum-II-type, by Definition \ref{def:Lind-type}\ref{def:LindII}, there exists an $\alpha$-saturated $\Delta\supseteq \Gamma$. This implies that $\Delta\nvdash\alpha$, i.e. $\alpha\notin C_{\vdash}(\Delta)$. So, $\alpha\notin \displaystyle\bigcap_{\substack{\Gamma\subseteq\Sigma\\\Sigma\in \mathbb{SAT}}}C_\vdash(\Sigma)$. Since $\alpha$ was chosen arbitrarily, we have $\Gamma\supseteq \displaystyle\bigcap_{\substack{\Gamma\subseteq\Sigma\\\Sigma\in \mathbb{SAT}}}C_\vdash(\Sigma)$. Now, as $(\mathscr{L},\vdash)$ is reflexive, $\Sigma\subseteq C_\vdash(\Sigma)$ for all $\Sigma\subseteq\mathscr{L}$. Thus 
\[
\Gamma\subseteq\bigcap_{\Gamma\subseteq\Sigma}\Sigma\subseteq\bigcap_{\Gamma\subseteq\Sigma}C_\vdash(\Sigma)\subseteq\bigcap_{\substack{\Gamma\subseteq\Sigma\\\Sigma\in \mathbb{SAT}}}C_\vdash(\Sigma).
\]
Hence $\Gamma=\displaystyle\bigcap_{\substack{\Gamma\subseteq\Sigma\\\Sigma\in \mathbb{SAT}}}C_\vdash(\Sigma)$.
\end{proof}

\begin{theorem}{\label{thm:cutT2Lind=>R}}
Let $(\mathscr{L},\vdash)$ be a Lindenbaum-II-type logical structure satisfying cut. Then it is reflexive.
\end{theorem}

\begin{proof}
Let $\Gamma\cup\{\alpha\}\subseteq \mathscr{L}$ such that $\alpha\in\Gamma$. If possible, suppose $\Gamma\nvdash\alpha$. Then, as $(\mathscr{L},\vdash)$ is of Lindenbaum-II-type, by Definition \ref{def:Lind-type}\ref{def:LindII}, there exists an $\alpha$-saturated $\Sigma\supseteq\Gamma$. So, $\Sigma\nvdash\alpha$. Now, since $\Sigma$ is saturated and $(\mathscr{L},\vdash)$ satisfies cut, by Theorem~\ref{thm:cut=>closed_sat}, $\Sigma$ is closed. Thus $\alpha\notin \Sigma$. This implies that $\alpha\notin \Gamma$. This is a contradiction. Hence $\Gamma\vdash\alpha$. Thus $(\mathscr{L},\vdash)$ is reflexive.
\end{proof}

\begin{theorem}{\label{thm:mcutT2Lind=>TL}}
Let $(\mathscr{L},\vdash)$ be a Lindenbaum-II-type logical structure satisfying mixed-cut. Then it is of Tarski-type.
\end{theorem}

\begin{proof}
Let $\Gamma\cup\{\alpha\}\subseteq \mathscr{L}$ such that $\Gamma\nvdash\alpha$. Since $(\mathscr{L},\vdash)$ is of Lindenbaum-II-type, by Definition \ref{def:Lind-type}\ref{def:LindII}, there exists an $\alpha$-saturated $\Sigma\supseteq\Gamma$. So, $\Sigma\nvdash\alpha$. Now, as $\Sigma$ is saturated and $(\mathscr{L},\vdash)$ satisfies mixed-cut, by Theorem~\ref{thm:mcut=>str-closed_sat}, $\Sigma$ is strongly closed. Thus $\Sigma$ is a strongly closed extension of $\Gamma$ such that $\Sigma\nvdash\alpha$. Thus by Theorem~\ref{thm:Char_Tar}, $(\mathscr{L},\vdash)$ is of Tarski-type.
\end{proof}

\subsection{Lindenbaum-IV-type logical structures}\label{subsec:LindIV}

We now prove the following characterization theorem for Lindenbaum-IV-type logical structures.

\begin{theorem}[\textsc{Characterization of Lindenbaum-IV}]\label{thm:Char_T4Lind}
Let $(\mathscr{L},\vdash)$ be a logical structure. Then the following statements are equivalent.
\begin{enumerate}[label=(\arabic*)]
\item $(\mathscr{L},\vdash)$ is of Lindenbaum-IV-type.
\item For all $\Gamma\cup\{\alpha\}\subseteq\mathscr{L}$, if $\Gamma\nvdash\alpha$ then there exists a maximal $\alpha$-saturated set $\Sigma\supseteq\Gamma$.
\item For all $\Gamma\cup\{\alpha\}\subseteq\mathscr{L}$, the set $\liml(\Gamma,\alpha):=\{\Sigma:\Gamma\subseteq \Sigma \hbox{ and }\Sigma\nvdash\alpha\}$ (i.e.\ the set of all extensions $\Sigma\supseteq\Gamma$ such that $\Sigma\nvdash\alpha$) has a maximal element whenever it is nonempty.
\end{enumerate}
\end{theorem}

\begin{proof}
We will use the following scheme to show that the above three statements are equivalent: (1)$\implies$(2); (2)$\implies$(3); (3)$\implies$(1).

\underline{(1)$\implies$(2)}: Suppose $\Gamma\cup\{\alpha\}\subseteq\mathscr{L}$ such that $\Gamma\nvdash\alpha$. Then, by Definition \ref{def:Lind-type}\ref{def:LindIV}, there exists $\Sigma\supseteq\Gamma$ which is relatively maximal in $\alpha$. So, $\Sigma$ is $\alpha$-saturated by Remark \ref{rem:alpha-sat}. Since $\Sigma$ is relatively maximal in $\alpha$, $\Delta\vdash\alpha$, and hence not $\alpha$-saturated, for every $\Delta\supsetneq\Sigma$. Thus $\Sigma$ is maximal $\alpha$-saturated.

\underline{(2)$\implies$(3)}: Suppose $\Gamma\cup\{\alpha\}\subseteq\mathscr{L}$ such that $\liml(\Gamma,\alpha)\ne\emptyset$. Let $\Sigma\in \liml(\Gamma,\alpha)$. Then, $\Sigma\nvdash\alpha$. So, by the assumed condition (2), there exists a maximal $\alpha$-saturated $\Delta\supseteq \Sigma$.  We claim that $\Delta$ is a maximal element of $\liml(\Gamma,\alpha)$.

Suppose the contrary. Then there exists $\Delta^\prime\in\liml(\Gamma,\alpha)$ such that $\Delta\subsetneq\Delta^\prime$. Now, since $\Delta^\prime\in\liml(\Gamma,\alpha)$, $\Delta^\prime\nvdash\alpha$. Again by the assumed condition (2), this implies that there exists a maximal $\alpha$-saturated $\Delta^{\prime\prime}\supseteq \Delta^\prime$. Thus $\Delta\subsetneq\Delta^{\prime\prime}$ and $\Delta^{\prime\prime}$ is $\alpha$-saturated. This contradicts the fact that $\Delta$ is maximal $\alpha$-saturated. Hence $\Delta$ must be a maximal element of $\liml(\Gamma,\alpha)$.

\underline{(3)$\implies$(1)}:
Let $\Gamma\cup\{\alpha\}\subseteq\mathscr{L}$ such that $\Gamma\nvdash\alpha$. Then $\Gamma\in\liml(\Gamma,\alpha)$, and hence $\liml(\Gamma,\alpha)\ne \emptyset$. So, by the assumed condition (3), $\liml(\Gamma,\alpha)$ has a maximal element, say $\Sigma$. Clearly, $\Gamma\subseteq\Sigma$. We claim that $\Sigma$ is also relatively maximal in $\alpha$. 

We note that $\Sigma\nvdash\alpha$ since $\Sigma\in\liml(\Gamma,\alpha)$. Now, suppose there exists $\Delta \supsetneq\Sigma$ such that $\Delta\nvdash\alpha$. Then, $\Delta\in\liml(\Gamma,\alpha)$. However, this contradicts the maximality of $\Sigma$ in $\liml(\Gamma,\alpha)$. Hence $\Delta\vdash\alpha$ for every $\Delta\supsetneq\Sigma$. Thus $\Sigma\supseteq\Gamma$ is relatively maximal in $\alpha$. Therefore, $(\mathscr{L},\vdash)$ is of Lindenbaum-IV-type.
\end{proof}

\begin{remark}
It follows from the above theorem that, if a logical structure $(\mathscr{L},\vdash)$ is of Lindenbaum-IV-type, then for all $\Gamma\cup\{\alpha\}\subseteq\mathscr{L}$ with $\Gamma\nvdash\alpha$, there exists an $\alpha$-saturated $\Sigma\supseteq\Gamma$. Now, the condition $\Gamma\nvdash\alpha$ implies that $\Gamma$ is nontrivial. As pointed out in Example \ref{exm:cpc_alpha-sat}, in case of classical propositional logic, every maximal consistent set $\Sigma$ is $\alpha$-saturated for all $\alpha\notin\Sigma$. Moreover, a set of classical propositional formulas is nontrivial iff it is consistent. So, in this special case, the statement (2) assumes the usual form of the Lindenbaum lemma for classical propositional logic, i.e.\ every consistent set of formulas can be extended to a maximal consistent set. Thus the above theorem provides a generalization of the usual Lindenbaum lemma for classical propositional logic.
\end{remark}

\begin{theorem}
Let $(\mathscr{L},\vdash)$ be a reflexive logical structure satisfying cut. Suppose further that for all $\Gamma\cup\{\alpha\}\subseteq\mathscr{L}$ with $\alpha\notin \Gamma$, if $C_\vdash(\Gamma\cup\{\alpha\})\subseteq C_\vdash(\Gamma)$ then $\Gamma$ is closed. Then $(\mathscr{L},\vdash)$ is of Lindenbaum-IV-type.
\end{theorem}

\begin{proof}
Let $\Gamma\cup\{\alpha\}\subseteq \mathscr{L}$ be such that  $\liml(\Gamma,\alpha)$ is nonempty. Let $\mathcal{C}$ be any chain in $\liml(\Gamma,\alpha)$ and $\Delta:=\displaystyle\bigcup_{\Sigma\in \mathcal{C}}\Sigma$. Clearly, $\Gamma\subseteq\Delta$. We claim that $\Delta\nvdash\alpha$. 

Suppose the contrary, i.e. $\Delta\vdash\alpha$. We first note that for each $\Sigma\in\mathcal{C}\subseteq\liml(\Gamma,\alpha)$, $\Sigma\nvdash\alpha$. Since $(\mathscr{L},\vdash)$ is reflexive, this implies that $\alpha\notin\Sigma$ for all $\Sigma\in\mathcal{C}$. Thus $\alpha\notin\Delta$. Now, let $\beta\in C_\vdash(\Delta\cup\{\alpha\})$, i.e. $\Delta\cup\{\alpha\}\vdash\beta$. Then by cut, we have $\Delta\vdash\beta$. So, $\beta\in C_\vdash(\Delta)$. Thus $C_\vdash(\Delta\cup\{\alpha\})\subseteq C_\vdash(\Delta)$. Hence by our hypothesis, $\Delta$ must be closed. Now, since $\Delta\vdash\alpha$, this implies that $\alpha\in\Delta$, which contradicts our earlier conclusion that $\alpha\notin\Delta$. So, $\Delta\nvdash\alpha$.

This implies that $\Delta\in\liml(\Gamma,\alpha)$. Hence $\Delta$ is an upper bound for $\mathcal{C}$ in $\liml(\Gamma,\alpha)$. Since $\mathcal{C}$ was an arbitrary chain in $\liml(\Gamma,\alpha)$, we can conclude that every chain in $\liml(\Gamma,\alpha)$ has an upper bound. So, by Zorn's Lemma, it follows that $\liml(\Gamma,\alpha)$ has a maximal element. Hence by Theorem~\ref{thm:Char_T4Lind}, it follows that $(\mathscr{L},\vdash)$ is of Lindenbaum-IV-type.
\end{proof}

We conclude this subsection with the following connection between the classes of Lindenbaum-IV- and Lindenbaum-II-type logical structures.

\begin{theorem}{\label{thm:MonLind=>(T2Lind<=>T4Lind)}}
Let $(\mathscr{L},\vdash)$ be a monotone logical structure. Then $(\mathscr{L},\vdash)$ is of Lindenbaum-IV-type iff it is of Lindenbaum-II-type.
\end{theorem}

\begin{proof}
Suppose $(\mathscr{L},\vdash)$ is a monotone logical structure. Then by Remark \ref{rem:alpha-sat}, any $\Sigma\subseteq\mathscr{L}$ is $\alpha$-saturated iff it is relatively maximal in $\alpha$. The theorem now follows from the definitions of Lindenbaum-IV and Lindenbaum-II-type structures.
\end{proof}

\subsection{Lindenbaum-I- and Lindenbaum-III-type logical structures}\label{subsec:LindI/III}

Some theorems of the previous subsections hold for Lindenbaum-I- and/or Lindenbaum-III-type logical structures as well, of course with suitable modifications. We mention these below. One may compare the Theorems \ref{thm:LindI} and \ref{thm:Char_T3Lind} below with the Theorems \ref{thm:LindII} and \ref{thm:Char_T4Lind}, respectively, from the previous subsections.

\begin{theorem}\label{thm:LindI}
Let $(\mathscr{L},\vdash)$ be a Lindenbaum-I-type logical structure. Then every maximal saturated set is maximal nontrivial.
\end{theorem}

\begin{proof}
Let $\Gamma\subseteq\mathscr{L}$ be a maximal saturated set. Then by Remark \ref{rem:sat=>nontriv}, $\Gamma$ is nontrivial. We claim that $\Gamma$ is, in fact, maximal nontrivial. Suppose the contrary. Then there exists a nontrivial $\Sigma\supsetneq\Gamma$. Now, as $(\mathscr{L},\vdash)$ is of Lindenbaum-I-type, by Definition \ref{def:Lind-type}\ref{def:LindI}, there exists a saturated $\Delta\supseteq\Sigma$. This implies that $\Delta\supsetneq\Gamma$, and hence contradicts the assumption that $\Gamma$ is a maximal saturated set. Thus $\Gamma$ must be maximal nontrivial.
\end{proof}

\begin{theorem}[\textsc{Characterization of Lindenbaum-III}]\label{thm:Char_T3Lind}
Let $(\mathscr{L},\vdash)$ be a logical structure. Then the following statements are equivalent.
\begin{enumerate}[label=(\arabic*)]
\item $(\mathscr{L},\vdash)$ is of Lindenbaum-III-type.
\item For all nontrivial $\Gamma\subseteq\mathscr{L}$, there exists a maximal saturated set $\Sigma\supseteq\Gamma$.
\item For all $\Gamma\subseteq\mathscr{L}$ the set $\liml(\Gamma):=\{\Sigma:\Gamma\subseteq \Sigma \hbox{ and }\Sigma\hbox{ is nontrivial}\}$ (i.e.\ the set of all nontrivial extensions $\Sigma\supseteq\Gamma$) has a maximal element whenever it is nonempty.
\end{enumerate}
\end{theorem}

\begin{proof}
We will use the following scheme to show that the above three statements are equivalent: (1)$\implies$(2); (2)$\implies$(3); (3)$\implies$(1).

\underline{(1)$\implies$(2)}: Suppose $\Gamma\subseteq\mathscr{L}$ is nontrivial. Then by Definition \ref{def:Lind-type}\ref{def:LindIII}, there exists a maximal nontrivial $\Sigma\supseteq\Gamma$. So, by Corollary \ref{cor:maxnontriv=>sat}, $\Sigma$ is saturated. We claim that $\Sigma$ is, in fact, maximal saturated. 

Suppose the contrary. Then there exists a saturated $\Delta\supsetneq\Sigma$. So, by Remark \ref{rem:sat=>nontriv}, $\Delta$ is nontrivial. Thus, again by Definition \ref{def:Lind-type}\ref{def:LindIII}, there exists a maximal nontrivial $\Delta^\prime\supseteq\Delta$. However, this implies that $\Delta^\prime\supsetneq\Sigma$. This contradicts the fact that $\Sigma$ is maximal nontrivial. Thus $\Sigma$ is maximal saturated.

\underline{(2)$\implies$(3)}: Suppose $\Gamma\subseteq\mathscr{L}$ such that $\liml(\Gamma)\neq\emptyset$. Let $\Sigma\in\liml(\Gamma)$. Then $\Gamma\subseteq\Sigma$ and $\Sigma$ is nontrivial. Hence by the assumed condition (2), there exists a maximal saturated $\Delta\supseteq\Sigma$. By Remark \ref{rem:sat=>nontriv}, $\Delta$ is nontrivial and since $\Gamma\subseteq\Sigma\subseteq\Delta$, $\Delta\in\liml(\Gamma)$. We claim that $\Delta$ is a maximal element of $\liml(\Gamma)$.

Suppose the contrary. Then there exists a $\Delta^\prime\supsetneq\Delta$ in $\liml(\Gamma)$. So, $\Delta^\prime$ is nontrivial, and hence again by the assumed condition (2), there exists a maximal saturated $\Delta^{\prime\prime}\supseteq\Delta^\prime$. This implies that $\Delta\subsetneq\Delta^{\prime\prime}$, which contradicts the assumption that $\Delta$ is maximal saturated. Thus $\Delta$ must be a maximal element of $\liml(\Gamma)$.

\underline{(3)$\implies$(1)}: Suppose $\Gamma\subseteq\mathscr{L}$ is nontrivial. Then $\Gamma\in\liml(\Gamma)$, and hence $\liml(\Gamma)\neq\emptyset$. Thus by the assumed condition (3), $\liml(\Gamma)$ must have a maximal element. Let $\Sigma\subseteq\mathscr{L}$ be such a maximal element of $\liml(\Gamma)$. Then $\Gamma\subseteq\Sigma$ and $\Sigma$ is nontrivial. So, $\Sigma$ is a maximal nontrivial set containing $\Gamma$. Hence $\mathscr{L}$ is of Lindenbaum-III-type.
\end{proof}

\begin{theorem}{\label{thm:mcutT1Lind=>pTL}}
Let $(\mathscr{L},\vdash)$ be a Lindenbaum-I-type logical structure satisfying mixed-cut. Then the set of trivial sets is upward closed.
\end{theorem}

\begin{proof}
Let $\Gamma\subseteq\Sigma\subseteq\mathscr{L}$ such that $\Gamma$ is trivial. If possible, suppose $\Sigma$ is nontrivial. Then, since $(\mathscr{L},\vdash)$ is of Lindenbaum-I-type, by Definition \ref{def:Lind-type}\ref{def:LindI}, there exists a saturated $\Delta\supseteq\Sigma$. So, by Remark \ref{rem:sat=>nontriv}, $\Delta$ is nontrivial. Now, as $(\mathscr{L},\vdash)$ satisfies mixed-cut, by Theorem~\ref{thm:mcut=>str-closed_sat}, $\Delta$ must be strongly closed. This, however, implies that $\Delta$ is trivial since $\Gamma$ is so and $\Gamma\subseteq\Delta$. This is a contradiction. Hence $\Sigma$ must be trivial. This proves that the set of trivial sets is upward closed.
\end{proof}

\subsection{Relationships between the Lindenbaum-types}\label{subsec:Lind-rel}

We end this section by discussing the relationships between the different Lindenbaum-types. We begin with the straightforward observations listed in the following theorem.

\begin{theorem}\label{thm:Lind}
Let $(\mathscr{L},\vdash)$ be a logical structure. 
\begin{enumerate}[label=(\arabic*)]
\item If $(\mathscr{L},\vdash)$ is of Lindenbaum-II-type then it is of Lindenbaum-I-type.
\item If $(\mathscr{L},\vdash)$ is of Lindenbaum-IV-type then it is of Lindenbaum-II-type, and hence also of Lindenbaum-I-type.
\item If $(\mathscr{L},\vdash)$ is of Lindenbaum-III-type then it is of Lindenbaum-I-type.
\end{enumerate}
\end{theorem}

\begin{proof}
Statement (1) follows from the facts that for every nontrivial set $\Gamma\subseteq\mathscr{L}$, there exists an $\alpha\in\mathscr{L}$ such that $\Gamma\nvdash\alpha$ and that every $\alpha$-saturated set is saturated.

Statement (2) follows from Remark \ref{rem:alpha-sat} and statement (1).

Statement (3) follows from Corollary \ref{cor:maxnontriv=>sat}.
\end{proof}

We devote the rest of this subsection to showing that none of the other possible inclusions between the classes of Lindenbaum-type logical structures holds.

\begin{example}[\textsc{Lindenbaum-I/II$\centernot\implies$Lindenbaum-III/IV}]\label{exm:LindI/II≠>LindIII/IV} 
Suppose $(\mathscr{L},\vdash)$ is a logical structure, where $\mathscr{L}$ is an infinite set and $\vdash\,\subseteq\,\mathcal{P}(\mathscr{L})\times\mathscr{L}$ be such that for any $\Gamma\subseteq\mathscr{L}$, 
\[
C_\vdash(\Gamma)=\left\{\begin{array}{ll}
\emptyset,&\hbox{if }\Gamma\hbox{ is finite and }\lvert\Gamma\rvert\hbox{ is odd},\\
\mathscr{L},&\hbox{otherwise}.\end{array}\right.
\]
We first show that $(\mathscr{L},\vdash)$ is of Lindenbaum-II-type. Let $\Gamma\cup\{\alpha\}\subseteq\mathscr{L}$ such that $\Gamma\nvdash\alpha$. Then $\Gamma$ must be nontrivial, and so $\Gamma$ must be finite with $\lvert\Gamma\rvert$ odd. So, for any $\beta\in\mathscr{L}\setminus\Gamma$ (such a $\beta$ exists since $\Gamma$ is finite while $\mathscr{L}$ is infinite), $\lvert\Gamma\cup\{\beta\}\rvert=\lvert\Gamma\rvert+1$ is even, and hence $\Gamma\cup\{\beta\}$ is trivial. So, in particular, $\Gamma\cup\{\beta\}\vdash\alpha$, for any $\beta\in\mathscr{L}\setminus\Gamma$. Hence $\Gamma$ is $\alpha$-saturated, i.e.\ $\Gamma$ is an $\alpha$-saturated set containing itself. Thus, $(\mathscr{L},\vdash)$ is of Lindenbaum-II-type, and hence also of Lindenbaum-I-type by Theorem \ref{thm:Lind}.

Next, to show that $(\mathscr{L},\vdash)$ is not of Lindenbaum-III-type, let $\Gamma\subseteq\mathscr{L}$ be nontrivial. If possible, let there be a maximal nontrivial $\Sigma\supseteq\Gamma$. Since $\Sigma$ is nontrivial, it must be finite with $\lvert\Sigma\rvert$ odd. Suppose $\beta,\gamma\in\mathscr{L}\setminus\Sigma$ with $\beta\neq\gamma$ (such $\beta,\gamma$ exist since $\Sigma$ is finite and $\mathscr{L}$ is infinite). Then $\lvert\Sigma\cup\{\beta,\gamma\}\rvert=\lvert\Sigma\rvert+2$ is also odd, and hence $\Sigma\cup\{\beta,\gamma\}$ is nontrivial. Thus $\Sigma$ is not maximal nontrivial. This is a contradiction. Hence no nontrivial set $\Gamma\subseteq\mathscr{L}$ is contained in a maximal nontrivial set. This implies that $(\mathscr{L},\vdash)$ cannot be of Lindenbaum-III-type.

Finally, to show that $(\mathscr{L},\vdash)$ is not of Lindenbaum-IV-type, let $\Gamma\cup\{\alpha\}\subseteq\mathscr{L}$ such that $\Gamma\nvdash\alpha$. If possible, suppose there exists a $\Sigma\supseteq\Gamma$ that is relatively maximal in $\alpha$. Then $\Sigma\nvdash\alpha$, which implies that $\Sigma$ is nontrivial, and hence finite with $\lvert\Sigma\rvert$ odd. Then, using the same argument as in the above proof of the fact that $(\mathscr{L},\vdash)$ is not of Lindenbaum-III-type, we have that $\Sigma\cup\{\beta,\gamma\}$ is nontrivial for $\beta,\gamma\in\mathscr{L}\setminus\Sigma$ with $\beta\neq\gamma$. Now, by definition of the logical structure $(\mathscr{L},\vdash)$, this implies that $C_\vdash(\Sigma\cup\{\beta,\gamma\})=\emptyset$, and hence in particular, $\Sigma\cup\{\beta,\gamma\}\nvdash\alpha$. This contradicts the assumption that $\Sigma$ is relatively maximal in $\alpha$. Hence, there cannot be any $\Gamma\cup\{\alpha\}\subseteq\mathscr{L}$ with $\Gamma\nvdash\alpha$, such that there exists a $\Sigma\supseteq\Gamma$ that is relatively maximal in $\alpha$. This implies that $(\mathscr{L},\vdash)$ cannot be of Lindenbaum-IV-type.

We can thus conclude that not every Lindenbaum-I/II-type logical structure is of Lindenbaum-III/IV-type.
\end{example}

\begin{remark}
It may be noted here that instead of having the finite subsets of odd cardinality as the only nontrivial subsets of $\mathscr{L}$, we could define $\vdash$ such that the only nontrivial sets are the ones in a collection $\mathcal{A}\subsetneq\{\Gamma\subseteq\mathscr{L}\mid\,\Gamma\hbox{ is finite}\}$, where $\mathcal{A}$ satisfies the following two conditions.
\begin{enumerate}[label=$\bullet$]
\item For every $\Gamma\in\mathcal{A}$ and $\alpha\notin\Gamma$, $\Gamma\cup\{\alpha\}\notin\mathcal{A}$.
\item For every finite $\Gamma\notin\mathcal{A}$ and $\alpha\notin\Gamma$, $\Gamma\cup\{\alpha\}\in\mathcal{A}$.
\end{enumerate}
Similar arguments, as in the above example, then lead to the same conclusions.
\end{remark}

\begin{example}[\textsc{Lindenbaum-III$\centernot\implies$Lindenbaum-II}]\label{exm:LindIII≠>LindII}
Let $(\mathscr{L},\vdash)$ be a logical structure, where $\mathscr{L}$ is an infinite set and $\vdash\,\subseteq\,\mathcal{P}(\mathscr{L})\times\mathscr{L}$ be such that for any $\Gamma\subseteq\mathscr{L}$, 
\[
C_{\vdash}(\Gamma)=\begin{cases}
\Lambda_0,&\hbox{if $\Gamma$ is finite},\\
\mathscr{L},&\hbox{if $\Gamma$ is infinite and }\Gamma\subsetneq\mathscr{L},
\\\mathscr{L}\setminus\Lambda_0,&\hbox{otherwise},
\end{cases}
\]
where $\emptyset\neq\Lambda_0\subsetneq\mathscr{L}$ is fixed. 

We first show that $(\mathscr{L},\vdash)$ is not of Lindenbaum-II-type. Let $\Gamma\cup\{\alpha\}\subseteq\mathscr{L}$ such that $\Gamma$ is finite and $\alpha\notin\Lambda_0$ (such an $\alpha$ exists since $\Lambda_0\subsetneq\mathscr{L}$). Then, by definition of $\vdash$, $\Gamma\nvdash\alpha$. If possible, suppose there exists an $\alpha$-saturated $\Sigma\supseteq\Gamma$. This implies that $\Sigma\nvdash\alpha$, and so, $\Sigma$ must be finite. Let $\beta\in\mathscr{L}\setminus\Sigma$ (again, such a $\beta$ exists as $\mathscr{L}$ is infinite). However, this implies that  $\Sigma\cup\{\beta\}$ is finite, and hence $\Sigma\cup\{\beta\}\nvdash\alpha$. This contradicts the assumption that $\Sigma$ is $\alpha$-saturated. Thus $\Gamma$ is not contained in any $\alpha$-saturated set. Hence $(\mathscr{L},\vdash)$ is not of Lindenbaum-II-type.  

Now, to show that $(\mathscr{L},\vdash)$ is of Lindenbaum-III-type, we first note that $\mathscr{L}\nvdash\beta$ for all $\beta\in\Lambda_0$, and hence is nontrivial. Clearly, $\mathscr{L}$ is also maximal nontrivial and contains every nontrivial subset of itself. 

We can thus conclude that not every Lindenbaum-III-type logical structure is of Lindenbaum-II-type.
\end{example}

\begin{remark}
It may be noted here that instead of using a fixed nonempty set $\Lambda_0\subsetneq\mathscr{L}$, as in the above example, we could define $C_\vdash(\Gamma)=\emptyset$ for each finite $\Gamma\subseteq\mathscr{L}$. This would lead to a different logical structure $(\mathscr{L},\vdash)$ that is also of Lindenbaum-III-type but not of Lindenbaum-II-type.
\end{remark}

\begin{corollary}
We can now use the results obtained thus far to conclude the following.
\begin{enumerate}[label=(\arabic*)]
    \item Not every Lindenbaum-I-type logical structure is of Lindenbaum-II-type.
    \item Not every Lindenbaum-III-type logical structure is of Lindenbaum-IV-type.
\end{enumerate}
\end{corollary}

\begin{proof}
\begin{enumerate}[label=(\arabic*)]
    \item Suppose the contrary. By Theorem \ref{thm:Lind}, every Lindenbaum-III-type logical structure is also of Lindenbaum-I-type. Thus, every Lindenbaum-III-type logical structure must also be of Lindenbaum-II-type. This, however, is not the case as shown by Example \ref{exm:LindIII≠>LindII}. Hence, not every Lindenbaum-I-type logical structure is of Lindenbaum-II-type.
    
    \item Suppose the contrary. Again, by Theorem \ref{thm:Lind}, every Lindenbaum-IV-type logical structure is also of Lindenbaum-II-type. Thus, every Lindenbaum-III-type logical structure must also be of Lindenbaum-II-type. This is not the case as shown by Example \ref{exm:LindIII≠>LindII}. Hence, not every Lindenbaum-III-type logical structure is of Lindenbaum-IV-type.
\end{enumerate}
\end{proof}

\begin{definition}\label{def:fin}
A logical structure $(\mathscr{L},\vdash)$ is said to be \emph{finitary} if for all $\Gamma\cup\{\alpha\}\subseteq\mathscr{L}$, whenever $\Gamma\vdash\alpha$, there exists a finite $\Gamma_0\subseteq\Gamma$ such that $\Gamma_0\vdash\alpha$.
\end{definition}

We next state the following well-known theorem without proof. 

\begin{theorem}[\textsc{Lindenbaum-Asser Theorem}]\label{thm:LA}
Let $(\mathscr{L},\vdash)$ be a finitary monotone logical structure. Then, for all $\Gamma\cup\{\alpha\}\subseteq\mathscr{L}$ with $\Gamma\nvdash\alpha$, there exists a $\Sigma\supseteq\Gamma$ that is relatively maximal in $\alpha$.
\end{theorem}

\begin{remark}
The proof of the above theorem requires the axiom of choice, as pointed out in \cite{Beziau2001}. In fact, the Lindenbaum-Asser Theorem (Theorem \ref{thm:LA}) is equivalent to the axiom of choice. For a proof of this result, see \cite{Dzik1981,Miller2007}.
\end{remark}

The following observation is now a straightforward corollary of Theorem \ref{thm:LA}.

\begin{corollary}\label{cor:finTarski=>LindIV}
Any finitary monotone logical structure is of Lindenbaum-IV-type, and hence, also of Lindenbaum-II/I-type, by Theorem \ref{thm:Lind}. Thus, any finitary Tarski-type logical structure is of Lindenbaum-IV-type, and hence also of Lindenbaum-II/I-type.
\end{corollary}

\begin{remark}
Thus, it follows that the underlying logical structure of any finitary Tarski-type logic is of Lindenbaum-IV-type. The prototypical example of such a logical structure is the one underlying classical propositional logic. The underlying logical structures of the logics mentioned in \cite{Arieli_Avron2015} are also of Lindenbaum-IV-type, and hence also of Lindenbaum-II/I-type.
\end{remark}

\begin{example}[\textsc{Lindenbaum-IV$\centernot\implies$Lindenbaum-III}]\label{exm:LindIVnimpLindIII}
We consider the logical structure $(\mathbb{N},\vdash)$, where $\mathbb{N}$ denotes the set of natural numbers, and $\vdash\,\subseteq\,\mathcal{P}(\mathbb{N})\times\mathbb{N}$ is such that, for any $\Gamma\subseteq\mathbb{N}$, 
\[
C_{\vdash}(\Gamma)=\begin{cases}
\mathbb{N},&\hbox{if $\Gamma$ is infinite},\\
\mathbb{N}\setminus\{n\},&\text{if $\Gamma$ is finite and }\lvert\Gamma\rvert=n.
\end{cases}
\]
We first show that $(\mathbb{N},\vdash)$ is of Lindenbaum-IV-type. Let $\Gamma\cup\{n\}\subseteq\mathbb{N}$ such that $\Gamma\nvdash n$. Then $\Gamma$ must be finite with $\lvert\Gamma\rvert=n$. Now, suppose $\Sigma\supsetneq\Gamma$. If $\Sigma$ is infinite, then by definition of $\vdash$, $\Sigma$ is trivial, and so, in particular, $\Sigma\vdash n$. On the other hand, if $\Sigma$ is finite, then since $\Sigma\supsetneq\Gamma$, $\lvert\Gamma\rvert=n<\lvert\Sigma\rvert$. Let $\lvert\Sigma\rvert=m$. Then, again by definition of $\vdash$, $C_\vdash(\Sigma)=\mathbb{N}\setminus\{m\}$. This implies that $n\in C_\vdash(\Sigma)$, and hence $\Sigma\vdash n$. Thus $\Sigma\vdash n$ for any $\Sigma\supsetneq\Gamma$. Hence $\Gamma$ is relatively maximal in $n$. In other words, $\Gamma$ is an extension of itself that is relatively maximal in $n$. So, we can conclude that $(\mathbb{N},\vdash)$ is a Lindenbaum-IV-type logical structure.

Now, to show that $(\mathbb{N},\vdash)$ is not of Lindenbaum-III-type, suppose the contrary. Let $\Gamma\subseteq\mathbb{N}$ be nontrivial. Then, there must exist a maximal nontrivial $\Sigma\supseteq\Gamma$. This implies that $\Sigma$ is nontrivial, and hence finite. Let $r\in\mathbb{N}\setminus\Sigma$ (such an $r$ exists, since $\Sigma$ is finite). Clearly, $\Sigma\cup\{r\}\supsetneq\Sigma$ is finite, and hence nontrivial. However, this contradicts the assumption that $\Sigma$ is maximal nontrivial. Hence $(\mathbb{N},\vdash)$ is not of Lindenbaum-III-type.
\end{example}

The results of this subsection have been summarized in Figure~\ref{fig:x}.

\begin{figure}[ht]
\centering
\begin{tikzcd}
                                 &                                  & \text{Lindenbaum-I}                \\
                                 &                                  &                                    \\
                                 & \text{Lindenbaum-II} \arrow[ruu] &                                    \\
                                 &                                  &                                    \\
\text{Lindenbaum-IV} \arrow[ruu] &                                  & \text{Lindenbaum-III} \arrow[uuuu]
\end{tikzcd}
%}
\caption{Comparison between Lindenbaum-type logical structures}\label{fig:x}
\end{figure}
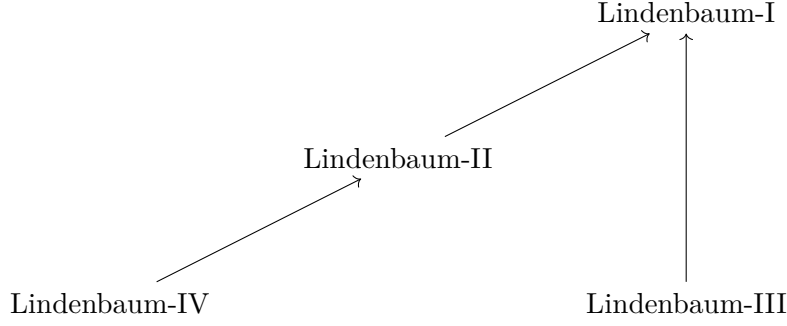

\FloatBarrier

\section{Tarski-type vis-\`a-vis Lindenbaum-type}\label{sec:T_R_L}

We have discussed the Tarski-type logical structures in Section \ref{sec:Tarski} and various Lindenbaum-type logical structures in Section \ref{sec:Lind}. A natural question that one may ask at this point is whether there is any dependency between these classes of logical structures.

It has been proved in \cite{Beziau1999} that not every Tarski-type logical structure is of Lindenbaum-I-type (weak Lindenbaum-Asser logic in the terminology there). Nevertheless, we give a different example below to illustrate this fact.

\begin{example}[\textsc{Tarski$\centernot\implies$Lindenbaum-I}]
Let $(\mathscr{L}_\mathbf{CL},\vdash_{\mathbf{CL}})$ denote classical propositional logic, as before. We define a new logical structure $(\lvert\mathscr{L}_\mathbf{CL}\rvert,\vdash_f)$, where $\vdash_f\subseteq\mathcal{P}\left(\lvert{\mathscr{L}_\mathbf{CL}\rvert}\right)\times\lvert{\mathscr{L}_\mathbf{CL}\rvert}$ is such that, for all $\Gamma\subseteq\lvert\mathscr{L}_\mathbf{CL}\rvert$,
\[
C_{\vdash_f}(\Gamma)=\begin{cases}
\Gamma,&\hbox{if }\Gamma\hbox{ is finite},\\
\lvert\mathscr{L}_\mathbf{CL}\rvert,&\hbox{otherwise}.
\end{cases}
\]
Clearly, $\lvert\mathscr{L}_\mathbf{CL}\rvert$ is infinite. Now, to show that $(\lvert\mathscr{L}_\mathbf{CL}\rvert,\vdash_f)$ is of Tarski-type, let $\Gamma,\Sigma\subseteq\lvert\mathscr{L}_\mathbf{CL}\rvert$. 

Suppose $\Sigma\subseteq C_{\vdash_f}(\Gamma)$. Now, if $\Gamma$ is finite, then $C_{\vdash_f}(\Gamma)=\Gamma$ is finite. So, $\Sigma$ must be finite. Thus in this case, $C_{\vdash_f}(\Sigma)=\Sigma\subseteq C_{\vdash_f}(\Gamma)$. On the other hand, if $\Gamma$ is infinite, then $C_{\vdash_f}(\Gamma)=\lvert\mathscr{L}_\mathbf{CL}\rvert$, and hence $C_{\vdash_f}(\Sigma)\subseteq C_{\vdash_f}(\Gamma)$.

Conversely, suppose $C_{\vdash_f}(\Sigma)\subseteq C_{\vdash_f}(\Gamma)$. Now, if $\Sigma$ is finite, then $C_{\vdash_f}(\Sigma)=\Sigma$. So, $\Sigma\subseteq C_{\vdash_f}(\Gamma)$ in this case. On the other hand, if $\Sigma$ is infinite, then $C_{\vdash_f}(\Sigma)=\lvert\mathscr{L}_\mathbf{CL}\rvert$, which implies that $C_{\vdash_f}(\Gamma)=\lvert\mathscr{L}_\mathbf{CL}\rvert$. Hence $\Sigma\subseteq C_{\vdash_f}(\Gamma)$.

Thus, for all $\Gamma,\Sigma\subseteq\lvert\mathscr{L}_\mathbf{CL}\rvert$, $\Sigma\subseteq C_{\vdash_f}(\Gamma)$ iff $C_{\vdash_f}(\Sigma)\subseteq C_{\vdash_f}(\Gamma)$. Hence, by using statement (4) of Theorem \ref{thm:Char_Tar}, we can conclude that $(\lvert\mathscr{L}_\mathbf{CL}\rvert,\vdash_f)$ is of Tarski-type.

We claim that $(\lvert\mathscr{L}_\mathbf{CL}\rvert,\vdash_f)$ is, however, not of Lindenbaum-I-type. Suppose the contrary. 

Let $\Gamma\subseteq\lvert\mathscr{L}_\mathbf{CL}\rvert$ be nontrivial. Then, by definition of Lindenbaum-I-type logical structures, there must exist a saturated $\Sigma\supseteq\Gamma$. Suppose $\Sigma$ is $\alpha$-saturated. So, $\Sigma\nvdash_f\alpha$, and hence $C_{\vdash_f}(\Sigma)\ne\lvert\mathscr{L}_\mathbf{CL}\rvert$. Then, $C_{\vdash_f}(\Sigma)=\Sigma$. Thus $\Sigma$ must be finite with $\alpha\notin\Sigma$. Now, let $\beta\in \lvert\mathscr{L}_\mathbf{CL}\rvert\setminus(\Sigma\cup\{\alpha\})$. Then, since $\Sigma\cup\{\beta\}$ is finite and $\alpha\notin\Sigma\cup\{\beta\}$, $\Sigma\cup\{\beta\}\nvdash_f\alpha$. This implies that $\Sigma$ is not $\alpha$-saturated, a contradiction. So, we can conclude that $\Sigma$ cannot be $\alpha$-saturated for any $\alpha\in\lvert\mathscr{L}_\mathbf{CL}\rvert$, and hence is not saturated. Thus $(\lvert\mathscr{L}_\mathbf{CL}\rvert,\vdash_f)$ is not of Lindenbaum-I-type.
\end{example}

\begin{remark}
We know, by Theorem \ref{thm:Lind}, that every logical structure which is of Lindenbaum-II/III/IV-type is of Lindenbaum-I-type as well. Thus, we can conclude, from the above example, that not every Tarski-type logical structure is of a Lindenbaum type. However, as pointed out in Corollary \ref{cor:finTarski=>LindIV}, every finitary Tarski-type logical structure is of Lindenbaum-IV/II/I-type. 
\end{remark}

\begin{example}[\textsc{Lindenbaum-III/IV$\centernot\implies$Tarski}]
Let $(\mathscr{L},\vdash)$ be a logical structure, where $\mathscr{L}=\{0,1,2\}$ and $\vdash$ is such that 
\[
C_\vdash(\{0\})=C_\vdash(\{0,1\})=\{0,1\},
\]
and $C_\vdash(\Gamma)=\mathscr{L}$ for any other $\Gamma\subseteq\mathscr{L}$.

We note that a set $\Gamma\subseteq \mathscr{L}$ is nontrivial iff $\Gamma\nvdash 2$ and $\{0,1\}$ is a maximal 2-saturated set that contains every nontrivial $\Gamma\subseteq\mathscr{L}$. Thus, by using statement (2) of Theorem \ref{thm:Char_T4Lind}, we can conclude that $(\mathscr{L},\vdash)$ is of Lindenbaum-IV-type. 

Moreover, $\{0,1\}$ is a maximal nontrivial set that extends every nontrivial $\Gamma\subseteq\mathscr{L}$. Hence $(\mathscr{L},\vdash)$ is also of Lindenbaum-III-type.

However, $\emptyset$ is trivial but is contained in the nontrivial sets $\{0\}$ and $\{0,1\}$. Thus $(\mathscr{L},\vdash)$ is not monotone, and hence not of Tarski-type. 
\end{example}

\begin{remark}
We can now conclude the following from the above counterexample.
\begin{enumerate}[label=(\arabic*)]
    \item Not every Lindenbaum-I-type logical structure is of Tarski-type.
    
    \noindent To prove this, suppose the contrary. By Theorem \ref{thm:Lind}, we know that every Lindenbaum-III-type logical structure is of Lindenbaum-I-type, and hence of Tarski-type. This, however, is not the case, as established by the above example. Hence, not every Lindenbaum-I-type logical structure is of Tarski-type.
    
    \item Not every Lindenbaum-II-type logical structure is of Tarski-type.
    
    \noindent To prove this, suppose the contrary. By Theorem \ref{thm:Lind}, we know that every Lindenbaum-IV-type logical structure is of Lindenbaum-II-type, and hence of Tarski-type. This, however, is not the case, as established by the above example. Hence, not every Lindenbaum-II-type logical structure is of Tarski-type.
\end{enumerate}
\end{remark}

The results of the current section and Subsection \ref{subsec:Lind-rel} can thus be summarized as follows. All five types of logical structures, viz., the four Lindenbaum-types and the Tarski-type are distinct.

\section{Tarski-Lindenbaum-type Logical Structures}\label{sec:Tar-Lind}
This section is devoted to discussing logical structures that are of Tarski-type as well as of a Lindenbaum-type.

\begin{definition}{\label{def:Tarski_Lind-type}}
A logical structure $(\mathscr{L},\vdash)$ is said to be of:

\begin{enumerate}[label=(\alph*)]
    \item \emph{Tarski-Lindenbaum-I-type} (henceforth abbreviated as $\tl{1}$-type) if it is of Tarski- as well as of Lindenbaum-I-type.
     \item \emph{Tarski-Lindenbaum-II-type} (henceforth abbreviated as $\tl{2}$-type) if it is of Tarski- as well as of Lindenbaum-II-type.
     \item \emph{Tarski-Lindenbaum-III-type} (henceforth abbreviated as $\tl{3}$-type) if it is of Tarski- as well as of Lindenbaum-III-type.
     \item \emph{Tarski-Lindenbaum-IV-type} (henceforth abbreviated as $\tl{4}$-type) if it is of Tarski- as well as of Lindenbaum-IV-type.
\end{enumerate}
\end{definition}

We first discuss the relationships between these classes of logical structures. To begin with, the following theorem lists some straightforward observations.

\begin{theorem}\label{thm:TarLind}
    Let $(\mathscr{L},\vdash)$ be a logical structure.
    \begin{enumerate}[label=(\arabic*)]
        \item If $(\mathscr{L},\vdash)$ is of $\tl{i}$-type, where $i=2,3,4$, then it is of $\tl{1}$-type.
        \item $(\mathscr{L},\vdash)$ is of $\tl{2}$-type iff it is of $\tl{4}$-type.
    \end{enumerate}
\end{theorem}

\begin{proof}
     Statement (1) follows from Theorem \ref{thm:Lind}.
     
     Statement (2) follows from Theorem \ref{thm:MonLind=>(T2Lind<=>T4Lind)}.
\end{proof}

The following examples now show that none of the other possible inclusions between the classes of Tarski-Lindenbaum-type logical structures holds. 

For the first two examples, we use the concept of ordinal numbers, more specifically, the von Neumann ordinal numbers. The definition and some important properties (without proof) of these are provided below for the readers' sake. These are based on the discussion in \cite[Chapter I, Section 7]{Kunen2011}. 

\begin{definition}[\textsc{von Neumann Ordinal Number}]\label{def:ord}
A set $\alpha$ is called a \emph{(von Neumann) ordinal number} or, simply an \emph{ordinal}, if it satisfies the following conditions\footnote{These conditions are usually stated in a compact form as follows. A transitive set well-ordered by $\in$ is a von Neumann ordinal. We avoid the terminology here as this is beside the main aim of the present article.}.
\begin{enumerate}[label=(\roman*)]
\item For all $x\in \alpha$, $x\subseteq \alpha$.
\item For all $x\in \alpha$, $x\notin x$.
\item For all $x,y,z\in \alpha$, if $x\in y$ and $y\in z$ then $x\in z$.
\item For all $x,y\in \alpha$ either $x\in y$, or $x=y$, or $y\in x$.
\item For every nonempty $x\subseteq\alpha$, there exists $z\in x$ such that $x\cap z=\emptyset$.
\end{enumerate}
\end{definition}

\begin{theorem}{\cite[Lemmas I.7.5, I.7.8, I.7.10, I.7.11, I.7.14]{Kunen2011}}\label{thm:ord_properties(I)}
Suppose $\alpha,\beta,\gamma$ are ordinals. Then the following statements hold.
\begin{enumerate}[label=(\arabic*)] 
    \item $\alpha\notin\alpha$.\label{thm:ord_irreflexive}
    \item If $\delta\in\alpha$, then $\delta$ is also an ordinal. \label{thm:ord_belong}
    \item $\beta\subseteq\alpha$ iff $\beta\in\alpha$ or $\beta=\alpha$. \label{thm:ord_subset_belong}
    \item $\alpha\in \beta$ and $\beta\in \gamma$ implies $\alpha\in \gamma$. \label{thm:ord_transitivity}
    \item Either $\alpha\in \beta$ or $\alpha=\beta$ or $\beta\in \alpha$. \label{thm:ord_trichotomy}
    \item $\alpha\cup\beta$ and $\alpha\cap\beta$ are ordinals. Moreover, if $X$ is a nonempty set of ordinals then $\displaystyle\bigcup X$ and $\displaystyle\bigcap X$ are  ordinals.\label{thm:uni_int}
\end{enumerate}
\end{theorem}

Lastly, the following are the von Neumann ordinals that we use in the following examples (see \cite{Kunen2011} for more details).
\begin{itemize}
    \item $0=\emptyset, 1=\{0\}, 2=\{0,1\},\cdots$, and in general, for any natural number $n$, $n+1=n\cup\{n\}$. These are the finite ordinals.
    \item $\omega=\{n\mid\,n\hbox{ is a natural number}\}=\mathbb{N}$ is the first infinite ordinal.
    \item $\omega+1=\omega\cup\{\omega\}$, and for any natural number $n$, $\omega+(n+1)=(\omega+n)\cup\{\omega+n\}$.
    \item $\omega+\omega=\displaystyle\bigcup_{n\in\omega}(\omega+n)$.
\end{itemize}

\begin{example}[$\tl{1}\centernot\implies\tl{2}/\tl{3}$]\label{exm:TL1notTL2/TL3}
We consider the logical structure $(\omega+\omega,\vdash)$, where $\vdash\,\subseteq\mathcal{P}(\omega+\omega)\times(\omega+\omega)$ is such that, for any $\Gamma\subseteq\omega+\omega$,
\[
C_{\vdash}(\Gamma)=\begin{cases}
\displaystyle\bigcap\{\beta\in\omega+\omega\mid\,\Gamma\subseteq \beta\},&\hbox{if}~\displaystyle\bigcap\{\beta\in\omega+\omega\mid\,\Gamma\subseteq \beta\}\neq\omega,\\
\omega+1,&\hbox{otherwise}.
\end{cases}
\]
We define, for our convenience, $\OC(\Gamma):=\{\beta\in \omega+\omega\mid\,\Gamma\subseteq \beta\}$, for each $\Gamma\subseteq\omega+\omega$. So, the above definition of $C_{\vdash}$ can be rephrased as follows.
\[
C_{\vdash}(\Gamma)=\begin{cases}
\displaystyle\bigcap\OC(\Gamma)&\hbox{if}~\displaystyle\bigcap\OC(\Gamma)\ne\omega\\\omega+1&\hbox{otherwise}.
\end{cases}
\]
Suppose $\Gamma\subseteq\omega+\omega$. Then the following are easy observations. 
\begin{enumerate}[label=(\alph*)]
\item $\OC(\Gamma)\neq\emptyset$.
    
\noindent Suppose $\Gamma=\emptyset$. Then $\OC(\Gamma)=\{\beta\in\omega+\omega\mid\,\emptyset=\Gamma\subseteq\beta\}=\omega+\omega\neq\emptyset$.
    
\noindent Now, suppose $\Gamma\neq\emptyset$. So, there exists $\alpha\in \Gamma$. Moreover, since $\alpha\in \Gamma\subseteq \omega+\omega$, by Theorem~\ref{thm:ord_properties(I)}\ref{thm:ord_belong}, $\alpha$ is an ordinal, i.e. $\Gamma$ is a set of ordinals. Then, by Theorem~\ref{thm:ord_properties(I)}\ref{thm:uni_int}, $\lambda=\displaystyle\bigcup_{\alpha\in \Gamma}\alpha$ is an ordinal. Now, for any $\beta\in\Gamma$, $\beta\subseteq\lambda$. This implies, by Theorem~\ref{thm:ord_properties(I)}\ref{thm:ord_subset_belong}, that either $\beta\in\lambda$ or $\beta=\lambda$. Thus, $\Gamma\subseteq\lambda\cup\{\lambda\}$, which is also an ordinal. So, $\lambda\cup\{\lambda\}\in \OC(\Gamma)$.
    
\noindent Hence, in all cases, $\OC(\Gamma)$ is nonempty.
\item $C_\vdash(\Gamma)$ is an ordinal in $\omega+\omega$.
    
\noindent Since $\OC(\Gamma)\neq\emptyset$, by Theorem~\ref{thm:ord_properties(I)}\ref{thm:uni_int}, $\displaystyle\bigcap\OC(\Gamma)$ is an ordinal. Clearly, $\displaystyle\bigcap\OC(\Gamma)$ is in $\omega+\omega$, since each member of $\OC(\Gamma)$ is in $\omega+\omega$. So, if $\displaystyle\bigcap\OC(\Gamma)\neq\omega$, then $C_\vdash(\Gamma)$ is an ordinal in $\omega+\omega$. On the other hand, if $\displaystyle\bigcap\OC(\Gamma)=\omega$, then $C_\vdash(\Gamma)=\omega+1$, also an ordinal in $\omega+\omega$.
\item $C_\vdash(\Gamma)\neq\omega$.
    
\noindent If $\displaystyle\bigcap\OC(\Gamma)=\omega$, then $C_\vdash(\Gamma)=\omega+1\neq\omega$. On the other hand, if $\displaystyle\bigcap\OC(\Gamma)\neq\omega$, then $C_\vdash(\Gamma)=\displaystyle\bigcap\OC(\Gamma)\neq\omega$.
\item $\Gamma\subseteq C_\vdash(\Gamma)$.
    
\noindent We first note that $\Gamma\subseteq\beta$ for all $\beta\in\OC(\Gamma)$ and so, $\Gamma\subseteq\displaystyle\bigcap\OC(\Gamma)$.
    
\noindent Now, if $\displaystyle\bigcap\OC(\Gamma)\neq\omega$, then $C_\vdash(\Gamma)=\displaystyle\bigcap\OC(\Gamma)$, and hence, $\Gamma\subseteq C_\vdash(\Gamma)$. On the other hand, if  $\displaystyle\bigcap\OC(\Gamma)=\omega$, then $\Gamma\subseteq\displaystyle\bigcap\OC(\Gamma)=\omega\subseteq\omega+1=C_\vdash(\Gamma)$.
\end{enumerate}

Now, let $\Gamma,\Sigma\subseteq \omega+\omega$. If $C_\vdash(\Gamma)\subseteq C_\vdash(\Sigma)$ then by (d) above, $\Gamma\subseteq C_\vdash(\Gamma)\subseteq C_\vdash(\Sigma)$. Conversely, suppose $\Gamma\subseteq C_\vdash(\Sigma)$. We first note that, by (b) above, $C_\vdash(\Sigma)$ is an ordinal in $\omega+\omega$. Thus $\Gamma\subseteq C_\vdash(\Sigma)$ implies that $C_\vdash(\Sigma)\in\OC(\Gamma)$, and hence $\displaystyle\bigcap\OC(\Gamma)\subseteq C_\vdash(\Sigma)$. 

Now, if $\displaystyle\bigcap\OC(\Gamma)\neq\omega$, then clearly, $C_\vdash(\Gamma)=\displaystyle\bigcap\OC(\Gamma)\subseteq C_{\vdash}(\Sigma)$.

On the other hand, if $\displaystyle\bigcap\OC(\Gamma)=\omega$, then $\omega=\displaystyle\bigcap\OC(\Gamma)\subseteq C_{\vdash}(\Sigma)$. As noted in (b) and (c) above, $C_\vdash(\Sigma)$ is an ordinal such that $C_\vdash(\Sigma)\neq\omega$. So, by Theorem \ref{thm:ord_properties(I)}\ref{thm:ord_subset_belong}, $\omega\in C_\vdash(\Sigma)$, which implies that $\{\omega\}\subseteq C_\vdash(\Sigma)$. Hence
\[
C_\vdash(\Gamma)=\omega+1=\omega\cup\{\omega\}\subseteq C_{\vdash}(\Sigma).
\]
Thus, for all $\Gamma,\Sigma\subseteq\omega+\omega$, $\Gamma\subseteq C_\vdash(\Sigma)$ iff $C_\vdash(\Gamma)\subseteq C_\vdash(\Sigma)$. So, by Theorem~\ref{thm:Char_Tar}, $(\omega+\omega,\vdash)$ is of Tarski-type.

Next, before proceeding to prove that $(\omega+\omega,\vdash)$ is of Lindenbaum-I-type, we establish the following.

\textsc{Claim:} For any $\alpha\in\omega+\omega$ such that $\alpha\neq\omega$, $\alpha$ is $\alpha$-saturated.\hspace{\fill}($*$)

We first note that, since $\alpha$ is an ordinal in $\omega+\omega$, $\alpha\in\OC(\alpha)$. Thus, $\alpha\subseteq\displaystyle\bigcap\OC(\alpha)\subseteq\alpha$, i.e. $\displaystyle\bigcap\OC(\alpha)=\alpha$. Now, since $\alpha\neq\omega$, $C_\vdash(\alpha)=\displaystyle\bigcap\OC(\alpha)=\alpha$. Then, as $\alpha\notin\alpha$, $\alpha\notin C_\vdash(\alpha)$, i.e. $\alpha\nvdash\alpha$.

Now, if possible, suppose there exists $\delta\in(\omega+\omega)\setminus\alpha$ such that $\alpha\cup\{\delta\}\nvdash\alpha$. If $\displaystyle\bigcap\OC(\alpha\cup\{\delta\})=\omega$, then we have $C_\vdash(\alpha\cup\{\delta\})=\omega+1$. This implies that $\alpha\subsetneq\alpha\cup\{\delta\}\subseteq C_\vdash(\alpha\cup\{\delta\})=\omega+1$. So, by Theorem \ref{thm:ord_properties(I)}\ref{thm:ord_subset_belong}, $\alpha\in\omega+1$, which means that $\alpha\in C_\vdash(\alpha\cup\{\delta\})$. This is a contradiction. Hence $\displaystyle\bigcap\OC(\alpha\cup\{\delta\})\neq\omega$, and so, $C_\vdash(\alpha\cup\{\delta\})=\displaystyle\bigcap\OC(\alpha\cup\{\delta\})$. Then, $\alpha\cup\{\delta\}\nvdash\alpha$ implies that $\alpha\notin\displaystyle\bigcap\OC(\alpha\cup\{\delta\})$. So, there exists $\beta\in\omega+\omega$ such that $\alpha\cup\{\delta\}\subseteq\beta$ but $\alpha\notin\beta$. Thus, by Theorem \ref{thm:ord_properties(I)}\ref{thm:ord_trichotomy}, either $\beta\in\alpha$ or $\beta=\alpha$, i.e. $\beta\subseteq\alpha$ by Theorem \ref{thm:ord_properties(I)}\ref{thm:ord_subset_belong}. This implies that $\alpha\cup\{\delta\}\subseteq\alpha$, which is impossible since $\delta\notin\alpha$. Thus we can conclude that, for every $\delta\in(\omega+\omega)\setminus\alpha$, $\alpha\cup\{\delta\}\vdash\alpha$. This completes the proof of the claim that, for every $\alpha\in\omega+\omega$, if $\alpha\neq\omega$ then $\alpha$ is $\alpha$-saturated. 

Now, to show that $(\omega+\omega,\vdash)$ is of Lindenbaum-I-type, let $\Gamma$ be a nontrivial subset of $\omega+\omega$. Then there exists $\alpha\in \omega+\omega$ such that $\Gamma\nvdash\alpha$. We have the following cases.

\textsc{Case 1:} $\displaystyle\bigcap\OC(\Gamma)=\omega$. 

Then $\Gamma\subseteq C_\vdash(\Gamma)=\omega+1$, and by Claim ($*$), $\omega+1$ is $(\omega+1)$-saturated, and hence saturated.

\textsc{Case 2:} $\displaystyle\bigcap\OC(\Gamma)\neq\omega$.

In this case, $C_\vdash(\Gamma)=\displaystyle\bigcap\OC(\Gamma)$. Then, since $\Gamma\nvdash\alpha$, i.e. $\alpha\notin C_\vdash(\Gamma)$, there exists $\beta\in \omega+\omega$ such that $\Gamma\subseteq \beta$ but $\alpha\notin \beta$. Then, by Theorem~\ref{thm:ord_properties(I)}\ref{thm:ord_trichotomy}, either $\beta\in \alpha$ or $\beta=\alpha$, i.e. $\beta\subseteq \alpha$ by Theorem~\ref{thm:ord_properties(I)}\ref{thm:ord_subset_belong}. This implies that $\Gamma\subseteq\alpha$.

Now, if $\alpha\neq\omega$, then $\alpha$ is $\alpha$-saturated, by Claim ($*$), and hence saturated. On the other hand, if $\alpha=\omega$, then $\Gamma\subseteq\omega\subseteq\omega+1$, and again by Claim ($*$), $\omega+1$ is $(\omega+1)$-saturated, and hence saturated.

So, in all cases, $\Gamma$ is contained in a saturated set. This implies that $(\omega+\omega,\vdash)$ is of Lindenbaum-I-type, and hence of $\tl{1}$-type.

Moreover, we note that in all the cases above, the nontrivial set $\Gamma$ has a proper saturated extension. This is clear in Case 1, where $\Gamma\subseteq\displaystyle\bigcap\OC(\Gamma)=\omega$, and in Case 2, where $\displaystyle\bigcap\OC(\Gamma)\neq\omega$ and $\Gamma\nvdash\omega$ -- in both scenarios, $\Gamma$ is properly contained in the saturated set, $\omega+1$. In the remaining case, where $\Gamma\nvdash\alpha$ and $\alpha\neq\omega$, $\alpha$ is an $\alpha$-saturated extension of $\Gamma$. If $\Gamma\subsetneq\alpha$, then we are done. Suppose not, i.e. $\Gamma=\alpha$. Now, since $\alpha\in\omega+\omega=\displaystyle\bigcup_{n\in\omega}(\omega+n)$, there exists $n\in\omega$ such that $\alpha\in\omega+n$. This implies that $\{\alpha\}\subseteq\omega+n$, and by definition of an ordinal, $\alpha\subseteq\omega+n$. So,
\[
\alpha\cup\{\alpha\}\subseteq\omega+n\subsetneq(\omega+n)\cup\{\omega+n\}=\omega+(n+1)\subseteq\omega+\omega.
\]
Thus $\alpha\cup\{\alpha\}\subsetneq\omega+\omega$, and by Theorem \ref{thm:ord_properties(I)}\ref{thm:uni_int}, it is an ordinal, often denoted by $\alpha+1$. Hence, by Theorem \ref{thm:ord_properties(I)}\ref{thm:ord_subset_belong} $\alpha\cup\{\alpha\}=\alpha+1\in\omega+\omega$. Now, if possible, suppose $\alpha+1=\omega$. Then $\alpha\in\omega$. This implies that $\alpha=n$ for some $n\in\omega$. Then $\alpha\in n+1$ This implies that $\{\alpha\}\subsetneq n+1$. Hence $\alpha\cup\{\alpha\}=\alpha+1\subseteq n+1\subsetneq n+2\subseteq\omega$. However, this means that $\omega=\alpha+1\subsetneq\omega$, which is impossible. Hence $\alpha+1\neq\omega$. So, by Claim ($*$), $\alpha+1$ is $(\alpha+1)$-saturated, and hence nontrivial. Now, $\Gamma=\alpha\subsetneq\alpha+1$. So, $\alpha+1$ is a proper saturated extension of $\Gamma$.

Hence every nontrivial $\Gamma\subseteq\omega+\omega$ has a proper saturated, and hence nontrivial, extension. This implies that there does not exist any maximal nontrivial subset of $\omega+\omega$. Therefore, we can conclude that $(\omega+\omega,\vdash)$ is not of Lindenbaum-III-type and thus not of $\tl{3}$-type.  

Finally, to show that $(\omega+\omega,\vdash)$ is not of Lindenbaum-II-type, and hence not of $\tl{2}$-type, we consider the finite ordinal 1, and note that $C_\vdash(1)=\displaystyle\bigcap\OC(1)=1$. Thus, $1\nvdash\omega$. Now, if possible, suppose $\Sigma\subseteq\omega+\omega$ be an $\omega$-saturated set containing 1. So, $\Sigma\nvdash\omega$, i.e. $\omega\notin C_\vdash(\Sigma)$. Since $C_\vdash(\Sigma)$ is an ordinal in $\omega+\omega$, as noted in Observation (b) above, this implies, by Theorem~\ref{thm:ord_properties(I)}\ref{thm:ord_trichotomy}, that either $C_\vdash(\Sigma)\in \omega$ or $C_\vdash(\Sigma)=\omega$. Since $C_\vdash(\Sigma)\neq\omega$, by Observation (c) above, $C_\vdash(\Sigma)\in\omega$. This implies that $C_\vdash(\Sigma)=n$ for some $n\in\omega$. Then, by reflexivity, $\Sigma\subseteq C_\vdash(\Sigma)=n$. So, $\Sigma\cup\{n\}\subseteq n\cup\{n\}=n+1$. Hence, by monotonicity, $C_\vdash(\Sigma\cup\{n\})\subseteq C_\vdash(n+1)=n+1$. Clearly, $\omega\notin C_\vdash(\Sigma\cup\{n\})$, i.e. $\Sigma\cup\{n\}\nvdash\omega$. This contradicts the assumption that $\Sigma$ is $\omega$-saturated. Thus, there is no $\omega$-saturated set containing 1, even though $1\nvdash\omega$. Thus $(\omega+\omega,\vdash)$ is not of Lindenbaum-II-type, and hence not of $\tl{2}$-type.
\end{example}

The next example uses a similar logical structure as in the above one; the only difference is in the definition of $\vdash$. We remove the separate cases now, and define $C_\vdash(\Gamma)$ in a uniform way for all $\Gamma\subseteq\omega+\omega$.

\begin{example}[$\tl{2}\centernot\implies\tl{3}$]\label{exm:TL2notTL3}
As in the above example, we consider a logical structure $(\omega+\omega,\vdash)$, where $\vdash\,\subseteq\mathcal{P}(\omega+\omega)\times(\omega+\omega)$ is such that, for any $\Gamma\subseteq\omega+\omega$,
\[
C_{\vdash}(\Gamma)=\displaystyle\bigcap\{\beta\in\omega+\omega\mid\,\Gamma\subseteq \beta\}=\displaystyle\bigcap\OC(\Gamma).
\]
It is fairly straightforward to see that $(\omega+\omega,\vdash)$ is of Tarski-type, but not of Lindenbaum-III-type, and hence not of $\tl{3}$-type, due to the same reasons as described in the previous example. We claim that $(\omega+\omega,\vdash)$ is, however, of Lindenbaum-II-type. 

To prove this, let $\Gamma\cup\{\alpha\}\subseteq  \omega+\omega$ be such that $\Gamma\nvdash\alpha$. Now, since $\alpha\in \omega+\omega$, it is an ordinal, and hence $\alpha\in\OC(\alpha)$. Thus, as in the previous example, $\alpha\subseteq C_\vdash(\alpha)=\displaystyle\bigcap\OC(\alpha)\subseteq\alpha$. Thus, $C_\vdash(\alpha)=\alpha$, and since $\alpha\notin\alpha$, $\alpha\notin C_\vdash(\alpha)$, i.e. $\alpha\nvdash\alpha$. We claim that $\alpha$ is, in fact, $\alpha$-saturated.

Suppose the contrary. Then there exists $\delta\in (\omega+\omega)\setminus\alpha$ such that $\alpha\cup\{\delta\}\nvdash\alpha$, i.e. $\alpha\notin C_\vdash(\alpha\cup\{\delta\})=\displaystyle\bigcap\OC(\alpha\cup\{\delta\})$. So, there exists $\beta\in\omega+\omega$ such that $\alpha\cup\{\delta\}\subseteq\beta$ but $\alpha\notin\beta$. Now, $\alpha\notin\beta$ implies, by Theorem \ref{thm:ord_properties(I)}\ref{thm:ord_trichotomy}, that either 
$\beta\in\alpha$ or $\beta=\alpha$, i.e. $\beta\subseteq\alpha$, by Theorem \ref{thm:ord_properties(I)}\ref{thm:ord_subset_belong}. Thus, we have $\alpha\cup\{\delta\}\subseteq\alpha$, which implies that $\delta\in\alpha$. This is a contradiction. Hence $\alpha$ is $\alpha$-saturated.

Thus, for any $\Gamma\cup\{\alpha\}\subseteq\omega+\omega$ with $\Gamma\nvdash\alpha$, there exists an $\alpha$-saturated extension of $\Gamma$, viz., $\alpha$. Hence $(\omega+\omega,\vdash)$ is of Lindenbaum-II-type, and consequently of $\tl{2}$-type.
\end{example}

\begin{example}[$\tl{3}\centernot\implies\tl{2}$]\label{exm:TL3notTL2}
We consider $(\mathbb{Z}^+,\vdash)$ as a logical structure, where $\mathbb{Z}^+$ is the set of positive integers, and $\vdash\,\subseteq\mathcal{P}(\mathbb{Z}^+)\times\mathbb{Z}^+$ is such that, for any $\Gamma\subseteq\mathbb{Z}^+$,
\[
C_{\vdash}(\Gamma)=\begin{cases}
\Gamma,&\hbox{if}~\Gamma\subseteq p\mathbb{Z}^+,~\hbox{for some prime}~p,\\
&\hbox{and is finite},\\
\displaystyle\bigcap\{q\mathbb{Z}^+\mid\,\Gamma\subseteq q\mathbb{Z}^+,\,q\hbox{ prime}\},&\hbox{if}~\Gamma\subseteq p\mathbb{Z}^+,~\hbox{for some prime}~p,\\
&\hbox{and is infinite},
\\\mathbb{Z}^+,&\hbox{otherwise}.
\end{cases}
\]
We first note that, for any $\Gamma\subseteq \mathbb{Z}^+$, $\Gamma\subseteq C_\vdash(\Gamma)$.
Now, to show that $(\mathbb{Z}^+,\vdash)$ is of Tarski-type, let $\Sigma,\Delta\subseteq \mathbb{Z}^+$ such that $C_\vdash(\Sigma)\subseteq C_\vdash(\Delta)$. Then, $\Sigma\subseteq C_\vdash(\Sigma)\subseteq C_\vdash(\Delta)$.

Conversely, suppose $\Sigma,\Delta\subseteq \mathbb{Z}^+$ such that $\Sigma\subseteq C_\vdash(\Delta)$. Then the following cases arise.

\textsc{Case 1:} $C_\vdash(\Delta)=\Delta$

This implies that $\Delta$ must be finite and $\Delta\subseteq p\mathbb{Z}^+$, for some prime $p$. Now, since $\Sigma\subseteq C_\vdash(\Delta)=\Delta$, it follows that $\Sigma$ is finite and $\Sigma\subseteq p\mathbb{Z}^+$ as well. Thus $C_\vdash(\Sigma)=\Sigma\subseteq C_\vdash(\Delta)$.

\textsc{Case 2:} $C_\vdash(\Delta)=\displaystyle\bigcap\{q\mathbb{Z}^+\mid\,\Delta\subseteq q\mathbb{Z}^+,\,q\hbox{ prime}\}$

Then it follows that $\Delta\subseteq p\mathbb{Z}^+$, for some prime $p$, and is infinite. So, $C_\vdash(\Delta)\subseteq p\mathbb{Z}^+$, and hence $\Sigma\subseteq C_\vdash(\Delta)\subseteq p\mathbb{Z}^+$. 

Now, if $\Sigma$ is finite, then $C_\vdash(\Sigma)=\Sigma\subseteq C_\vdash(\Delta)$.

If, on the other hand, $\Sigma$ is infinite, $C_\vdash(\Sigma)=\displaystyle\bigcap\{q\mathbb{Z}^+\mid\,\Sigma\subseteq q\mathbb{Z}^+,\,q\hbox{ prime}\}$. Then, since $\Sigma\subseteq C_\vdash(\Delta)$, $\Sigma\subseteq q\mathbb{Z}^+$, for every prime $q$ such that $\Delta\subseteq q\mathbb{Z}^+$. This implies that 
\[
\begin{array}{lcl}
C_\vdash(\Sigma)&=&\displaystyle\bigcap\{q\mathbb{Z}^+\mid\,\Sigma\subseteq q\mathbb{Z}^+,\,q\hbox{ prime}\}\\
&\subseteq&\displaystyle\bigcap\{q\mathbb{Z}^+\mid\,\Delta\subseteq q\mathbb{Z}^+,\,q\hbox{ prime}\}\\
&=&C_\vdash(\Delta).
\end{array}
\]

\textsc{Case 3:} $C_\vdash(\Delta)=\mathbb{Z}^+$

In this case, it immediately follows that $C_\vdash(\Sigma)\subseteq C_\vdash(\Delta)$.

Thus, in all cases, $\Sigma\subseteq C_\vdash(\Delta)$ implies that $C_\vdash(\Sigma)\subseteq C_\vdash(\Delta)$. Hence, for any $\Sigma,\Delta\subseteq\mathbb{Z}^+$, $\Sigma\subseteq C_\vdash(\Delta)$ iff $C_\vdash(\Sigma)\subseteq C_\vdash(\Delta)$, and so, by Theorem~\ref{thm:Char_Tar}, $(\mathbb{Z}^+,\vdash)$ is of Tarski-type.

We next show that $(\mathbb{Z}^+,\vdash)$ is of Lindenbaum-III-type. Suppose $\Sigma\subseteq\mathbb{Z}^+$ is nontrivial. Then there must exist a prime $p$ such that $\Sigma\subseteq p\mathbb{Z}^+$. We claim that $p\mathbb{Z}^+$ is maximal nontrivial. 

Suppose the contrary. Let $\Delta\supsetneq p\mathbb{Z}^+$ be a nontrivial subset of $\mathbb{Z}^+$. This implies that there exists a prime $p^\prime$ such that $\Delta\subseteq p^\prime\mathbb{Z}^+$. So, we have $p\mathbb{Z}^+\subsetneq\Delta\subseteq p^\prime\mathbb{Z}^+$. Then $p^\prime\mid p$. Since $p,p^\prime$ are both prime, this implies that $p=p^\prime$, and hence $p\mathbb{Z}^+=p^\prime\mathbb{Z}^+$. This is, however, a contradiction. Hence $p\mathbb{Z}^+$ is maximal nontrivial. Thus every nontrivial $\Sigma\subseteq\mathbb{Z}^+$ has a maximal nontrivial extension. So, $(\mathbb{Z}^+,\vdash)$ is of Lindenbaum-III-type, and hence of $\tl{3}$-type.

Lastly, to show that $(\mathbb{Z}^+,\vdash)$ is not of Lindenbaum-II-type, and hence not of $\tl{2}$-type, consider the set $\{2\}\subseteq2\mathbb{Z}^+\subseteq\mathbb{Z}^+$. Then as $\{2\}$ is finite, $C_\vdash(\{2\})=\{2\}$, and so, in particular, $\{2\}\nvdash 4$. If possible, suppose there exists a 4-saturated $\Delta\supseteq\{2\}$. 

Then $\Delta\nvdash4$, and hence is nontrivial. So, there must exist a prime $q$ such that $\Delta\subseteq q\mathbb{Z}^+$. Since $2\in\Delta$, $q$ must be 2, i.e. $\Delta\subseteq2\mathbb{Z}^+$. Moreover, we claim that $\Delta$ is infinite. Suppose the contrary. Then $C_\vdash(\Delta)=\Delta$, and so, $4\notin \Delta$. Now, let $k\in\mathbb{Z}^+$ such that $k>2$ and $2k\in\mathbb{Z}^+\setminus\Delta$ (clearly, such a $k$ exists as $\Delta$ is finite). Then $\Delta\cup\{2k\}\subseteq 2\mathbb{Z}^+$ is finite, and so, $C_\vdash(\Delta\cup\{2k\})=\Delta\cup\{2k\}$. Since $2k\neq4$ and $4\notin\Delta$, this implies that $\Delta\cup\{2k\}\nvdash 4$. This, however, contradicts the assumption that $\Delta$ is $4$-saturated. So, $\Delta$ must be infinite. Then, $C_\vdash(\Delta)=\displaystyle\bigcap\{q\mathbb{Z}^+\mid\,\Delta\subseteq q\mathbb{Z}^+,\,q\hbox{ prime}\}$. Now, since $2\in\Delta$, we must have $2\in q\mathbb{Z}^+$ for every prime $q$ such that $\Delta\subseteq q\mathbb{Z}^+$, which implies that $q=2$. Thus $C_\vdash(\Delta)=\displaystyle\bigcap\{q\mathbb{Z}^+\mid\,\Delta\subseteq q\mathbb{Z}^+,\,q\hbox{ prime}\}=2\mathbb{Z}^+$. However, this implies that $\Delta\vdash 4$, a contradiction. Hence $\{2\}\nvdash4$ but $\{2\}$ is not contained in a $4$-saturated set. Thus $(\mathbb{Z}^+,\vdash)$ is not of Lindenbaum-II-type, and hence not of $\tl{2}$-type.
\end{example}

\begin{corollary}
    We can now use the results obtained thus far to conclude the following.
    \begin{enumerate}[label=(\arabic*)]
        \item Not every $\tl{1}$-type structure is of $\tl{4}$-type.
        \item Not every $\tl{3}$-type structure is of $\tl{4}$-type.
        \item Not every $\tl{4}$-type structure is of $\tl{3}$-type.
    \end{enumerate}
\end{corollary}

\begin{proof}
    \begin{enumerate}[label=(\arabic*)]
        \item Suppose the contrary. By Theorem \ref{thm:TarLind}, every $\tl{4}$-type logical structure is of $\tl{2}$-type. Thus every $\tl{1}$-type logical structure is of $\tl{2}$-type. This, however, is not the case as shown by Example \ref{exm:TL1notTL2/TL3}. Hence, not every $\tl{1}$-type structure is of $\tl{4}$-type.
        \item Suppose the contrary. Again, by Theorem \ref{thm:TarLind}, every $\tl{4}$-type logical structure is of $\tl{2}$-type. Thus every $\tl{3}$-type logical structure is of $\tl{2}$-type. This, however, is not the case as shown by Example \ref{exm:TL3notTL2}. Hence, not every $\tl{3}$-type structure is of $\tl{4}$-type.
        \item Suppose the contrary. Again, by Theorem \ref{thm:TarLind}, every $\tl{2}$-type logical structure is of $\tl{4}$-type. Thus every $\tl{2}$-type logical structure is of $\tl{3}$-type. This, however, is not the case as shown by Example \ref{exm:TL2notTL3}. Hence, not every $\tl{4}$-type structure is of $\tl{3}$-type.
    \end{enumerate}
\end{proof}

The results in this section have been summarized in Figure~\ref{fig:y}.

\begin{figure}[ht]
\centering
\begin{tikzcd}
                              &                                           & \tl{1}              \\
                              &                                           &                     \\
                              & \tl{2} \arrow[ruu] \arrow[ldd,<->] &                     \\
                              &                                           &                     \\
\tl{4}  &                                           & \tl{3} \arrow[uuuu]
\end{tikzcd}
\caption{Comparison between Tarski-Lindenbaum-type logical structures}
\label{fig:y}
\end{figure}
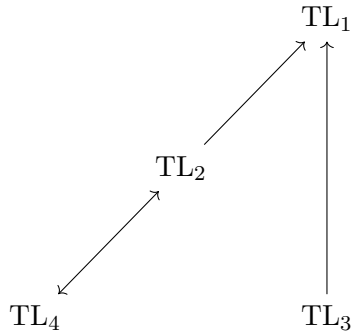

\section{\texorpdfstring{$\tl{4}$}{TL-4}-Logical Structures}\label{sec:TL-4}
We devote this section to the study of $\tl{4}$-type logical structures. Since a logical structure is of $\tl{4}$-type iff it is of $\tl{2}$-type, by Theorem~\ref{thm:MonLind=>(T2Lind<=>T4Lind)}, the results of this section apply to $\tl{2}$-type logical structures as well.

\begin{theorem}{\label{thm:Tar=>mcut}}
Let $(\mathscr{L},\vdash)$ be a Tarski-type logical structure. Then it satisfies mixed-cut.
\end{theorem}

\begin{proof}
Follows straightforwardly from the definition of Tarski-type logical structures. 
\end{proof}

\begin{theorem}[\textsc{Characterization of $\tl{4}$}]\label{thm:Char_TL4}
Let $(\mathscr{L},\vdash)$ be a logical structure. Then the following statements are equivalent.

\begin{enumerate}[label=(\arabic*)]
\item $(\mathscr{L},\vdash)$ is of $\tl{4}$-type.
\item For all $\Gamma\cup\{\alpha\}\subseteq\mathscr{L}$ with $\Gamma\nvdash\alpha$, there exists a strongly closed $\Sigma\supseteq\Gamma$ which is relatively maximal in $\alpha$.
\item For all $\Gamma\cup\{\alpha\}\subseteq\mathscr{L}$ with $\Gamma\nvdash\alpha$, there exists a strongly closed $\alpha$-saturated $\Sigma\supseteq\Gamma$.
\item For all $\Gamma\cup\{\alpha\}\subseteq\mathscr{L}$ with $\Gamma\nvdash\alpha$, there exists $\beta\in\mathscr{L}$ and a strongly closed $\beta$-saturated $\Sigma\supseteq\Gamma$ such that $\Sigma\nvdash\alpha$ and $\{\beta\}\vdash\alpha$.
\item For each $\alpha\in\mathscr{L}$, there exists $\beta\in\mathscr{L}$ such that, for any $\Gamma\subseteq\mathscr{L}$, if $\Gamma\nvdash\alpha$ then there exists a strongly closed $\beta$-saturated $\Sigma\supseteq\Gamma$ with $\Sigma\nvdash\alpha$.
\end{enumerate}
\end{theorem}

\begin{proof}
We will use the following scheme to show that the above statements are equivalent: 
\[
(1)\implies(2)\implies(3)\implies(4)\implies(3)\implies(5)\implies(1).
\]
\underline{(1)$\implies$(2)}: Let $\Gamma\cup\{\alpha\}\subseteq\mathscr{L}$ be such that $\Gamma\nvdash\alpha$. Since $(\mathscr{L},\vdash)$ is of Lindenbaum-IV-type, there exists $\Sigma\supseteq\Gamma$ which is relatively maximal in $\alpha$. Now, by Remark \ref{rem:alpha-sat}, $\Sigma$ is also $\alpha$-saturated, and hence saturated. Since $(\mathscr{L},\vdash)$ is of Tarski-type as well, it satisfies mixed-cut, by Theorem~\ref{thm:Tar=>mcut}. Then, by Theorem~\ref{thm:mcut=>str-closed_sat}, $\Sigma$ is strongly closed. Thus $\Sigma$ is a strongly closed extension of $\Gamma$ that is relatively maximal in $\alpha$.

\underline{(2)$\implies$(3)}: This follows immediately, by Remark \ref{rem:alpha-sat}.

\underline{(3)$\implies$(4)}: Suppose $\Gamma\cup\{\alpha\}\subseteq\mathscr{L}$ such that $\Gamma\nvdash\alpha$. Then, by statement (3), there exists a strongly closed $\alpha$-saturated $\Sigma\supseteq\Gamma$. Since $\Sigma$ is $\alpha$-saturated, $\Sigma\nvdash\alpha$. Thus, by Theorem \ref{thm:Char_Tar}, $(\mathscr{L},\vdash)$ is of Tarski-type, and hence reflexive. So, $\{\alpha\}\vdash\alpha$. The statement (4) now follows with $\beta=\alpha$.

\underline{(4)$\implies$(3)}:
Suppose $\Gamma\cup\{\alpha\}\subseteq\mathscr{L}$ such that $\Gamma\nvdash\alpha$. Then, by statement (4), there exists $\beta\in\mathscr{L}$ and a strongly closed $\beta$-saturated $\Sigma\supseteq\Gamma$ such that $\Sigma\nvdash\alpha$ and $\{\beta\}\vdash\alpha$. Thus, in particular, $\Sigma$ is a strongly closed extension of $\Gamma$ such that $\Sigma\nvdash\alpha$. So, by Theorem \ref{thm:Char_Tar}, $(\mathscr{L},\vdash)$ is of Tarski-type. Now, let $\gamma\in\mathscr{L}\setminus\Sigma$. Since $\Sigma$ is $\beta$-saturated, $\Sigma\cup\{\gamma\}\vdash\beta$. Then, as $\{\beta\}\vdash\alpha$, we have, by transitivity, $\Sigma\cup\{\gamma\}\vdash\alpha$. Thus $\Sigma$ is $\alpha$-saturated as well.

\underline{(3)$\implies$(5)}: The statement (5) follows from (3) by taking $\beta=\alpha$.

\underline{(5)$\implies$(1)}: Suppose $\Gamma\cup\{\alpha\}\subseteq\mathscr{L}$ such that $\Gamma\nvdash\alpha$. Then, by statement (5), there exists a $\beta\in\mathscr{L}$ (dependent only on $\alpha$), and a strongly closed $\beta$-saturated $\Sigma\supseteq\Gamma$ with $\Sigma\nvdash\alpha$. Thus, in particular, $\Sigma$ is a strongly closed extension of $\Gamma$ such that $\Sigma\nvdash\alpha$. So, by Theorem \ref{thm:Char_Tar}, $(\mathscr{L},\vdash)$ is of Tarski-type.

Next, to show that $(\mathscr{L},\vdash)$ is of Lindenbaum-IV-type as well, let $\Gamma\cup\{\alpha\}\subseteq\mathscr{L}$ such that $\Gamma\nvdash\alpha$. So, by the assumed condition, there exists a $\beta\in\mathscr{L}$ (dependent only on $\alpha$), and a strongly closed $\beta$-saturated $\Sigma\supseteq\Gamma$ with $\Sigma\nvdash\alpha$. If possible, suppose $\Sigma$ is not relatively maximal in $\alpha$. Then there exists a $\Delta\supsetneq\Sigma$ such that $\Delta\nvdash\alpha$. So, again by the assumed condition, there exists a strongly closed $\beta$-saturated $\Lambda\supseteq\Delta$ with $\Lambda\nvdash\alpha$. We note that $\Sigma\subsetneq\Delta\subseteq\Lambda$. So, there exists a $\gamma\in\Lambda\setminus\Sigma$. Now, since $\Sigma$ is $\beta$-saturated, $\Sigma\cup\{\gamma\}\vdash\beta$. Then, as $\Lambda$ is strongly closed and $\Sigma\cup\{\gamma\}\subseteq\Lambda$,  $\Lambda\vdash\beta$. This, however, contradicts that $\Lambda$ is $\beta$-saturated. Thus $\Sigma$ must be relatively maximal in $\alpha$. Hence $(\mathscr{L},\vdash)$ is of Lindenbaum-IV-type, and consequently of $\tl{4}$-type.
\end{proof}

\begin{lemma}\label{lem:LS_soundness}
Let $(\mathscr{L},\vdash)$ be a logical structure and $\mathbb{SCS}$ (strongly closed and saturated) be a set of bivaluations defined as follows.
\[
\mathbb{SCS}=\displaystyle\bigcup_{\beta\in \mathscr{L}}\{\chi_\Sigma\in\{0,1\}^\mathscr{L}:\Sigma\hbox{ is strongly closed and $\beta$-saturated}\},
\]
where $\chi_{\Sigma}$ is the characteristic function of $\Sigma$. Then, for all $\Gamma\cup\{\alpha\}\subseteq\mathscr{L}$, $\Gamma\vdash\alpha$ implies that $\Gamma\vdash_{\mathbb{SCS}}\alpha$.
\end{lemma}

\begin{proof}
Suppose $\Gamma\cup\{\alpha\}\subseteq\mathscr{L}$ such that $\Gamma\vdash\alpha$. If possible, let $\Gamma\nvdash_{\mathbb{SCS}}\alpha$. This implies that there exists a $v\in \mathbb{SCS}$ such that $v(\Gamma)=\{1\}$ but $v(\alpha)=0$. Then, by the definition of $\mathbb{SCS}$, there exists $\Sigma\cup\{\beta\}\subseteq\mathscr{L}$ such that $\Sigma$ is a strongly closed $\beta$-saturated set and $v=\chi_{\Sigma}$. Now, $\chi_{\Sigma}(\Gamma)=\{1\}$ implies that $\Gamma\subseteq \Sigma$. Then, as $\Sigma$ is strongly closed and $\Gamma\vdash\alpha$, it follows that $\alpha\in\Sigma$. Hence $\chi_{\Sigma}(\alpha)=v(\alpha)=1$. This is a contradiction. Hence $\Gamma\vdash_{\mathbb{SCS}}\alpha$.
\end{proof}

\begin{corollary}[\textsc{Soundness for $\tl{4}$}]\label{cor:TL_sound}
Let $(\mathscr{L},\vdash)$ be a $\tl{4}$-type logical structure and $\mathbb{SCS}$ be as in the above lemma. Then $\mathbb{SCS}\neq\emptyset$ (assuming that there is at least one nontrivial $\Gamma\subseteq\mathscr{L}$) and $(\mathscr{L},\vdash)$ is sound with respect to $(\mathscr{L},\vdash_\mathbb{SCS})$.
\end{corollary}

\begin{lemma}[\textsc{Completeness for $\tl{4}$}]\label{lem:TL_complete}
Let $(\mathscr{L},\vdash)$ be a $\tl{4}$-type logical structure and $\mathbb{SCS}$ be as defined in Lemma \ref{lem:LS_soundness}. Then, for all $\Gamma\cup\{\alpha\}\subseteq\mathscr{L}$, $\Gamma\vdash_{\mathbb{SCS}}\alpha$ implies that $\Gamma\vdash\alpha$, i.e.\ $(\mathscr{L},\vdash)$ is complete with respect to $(\mathscr{L},\vdash_\mathbb{SCS})$.
\end{lemma}

\begin{proof}
Let $\Gamma\cup\{\alpha\}\subseteq\mathscr{L}$ such that $\Gamma\vdash_{\mathbb{SCS}}\alpha$ but $\Gamma\nvdash\alpha$. Since $(\mathscr{L},\vdash)$ is of $\tl{4}$-type, this implies, by statement (5) of Theorem \ref{thm:Char_TL4}, that there exists a $\beta\in\mathscr{L}$ (dependent only on $\alpha$) and a strongly closed $\beta$-saturated $\Sigma\supseteq\Gamma$ with $\Sigma\nvdash\alpha$. Now, as $\Sigma$ is strongly closed and $\beta$-saturated, $\chi_\Sigma\in\mathbb{SCS}$, and since $\Gamma\subseteq\Sigma$, $\chi_\Sigma(\Gamma)=\{1\}$. Then, using the assumption $\Gamma\vdash_\mathbb{SCS}\alpha$, we have $\chi_\Sigma(\alpha)=1$, which implies that $\alpha\in\Sigma$. Now, since $\Sigma$ is strongly closed, it follows that $\Sigma\vdash\alpha$. This is a contradiction. Hence, if $\Gamma\vdash_\mathbb{SCS}\alpha$ then $\Gamma\vdash\alpha$.
\end{proof}

\begin{theorem}[\textsc{Adequacy Theorem for $\tl{4}$}]\label{thm:Adeq_TL}
Let $(\mathscr{L},\vdash)$ be a logical structure. If $(\mathscr{L},\vdash)$ is of $\tl{4}$-type then $\vdash\,=\,\vdash_{\mathbb{SCS}}$.
\end{theorem}

\begin{proof}
Suppose $(\mathscr{L},\vdash)$ is a $\tl{4}$-type logical structure. Then it follows that $\vdash\,=\,\vdash_\mathbb{SCS}$ from Corollary \ref{cor:TL_sound} and Lemma \ref{lem:TL_complete}.
\end{proof} 

\begin{theorem}[\textsc{Minimality Theorem for $\tl{4}$}]\label{thm:Min_TL}
Let $(\mathscr{L},\vdash)$ be a $\tl{4}$-type logical structure. Then, for any $\mathcal{B}\subseteq\mathbb{SCS}$, $\vdash\,\subseteq\,\vdash_{\mathcal{B}}$. Moreover, if $\mathcal{B}\subsetneq\mathbb{SCS}$ then  $\vdash\,\subsetneq\,\vdash_{\mathcal{B}}$.
\end{theorem}

\begin{proof}
Suppose $\mathcal{B}\subseteq \mathbb{SCS}$. We claim that this implies $\vdash_\mathbb{SCS}\,\subseteq\,\vdash_\mathcal{B}$. Suppose the contrary. Then, there exists $\Gamma\cup\{\alpha\}\subseteq\mathscr{L}$ such that $\Gamma\vdash_\mathbb{SCS}\alpha$, while $\Gamma\nvdash_\mathcal{B}\alpha$. This implies that, there exists $v\in\mathcal{B}$ such that $v(\Gamma)=\{1\}$ but $v(\alpha)\neq1$. Now, since $\mathcal{B}\subseteq\mathbb{SCS}$, $v\in\mathbb{SCS}$. Then $v(\Gamma)=\{1\}$ while $v(\alpha)\neq1$ implies that $\Gamma\nvdash_\mathbb{SCS}\alpha$. This is a contradiction. Hence $\vdash_\mathbb{SCS}\,\subseteq\,\vdash_\mathcal{B}$.

Now, suppose $\mathcal{B}\subsetneq \mathbb{SCS}$. Then, there exists $\beta\in\mathscr{L}$ and a strongly closed $\beta$-saturated $\Sigma\subseteq\mathscr{L}$ such that $\chi_\Sigma\in \mathbb{SCS}$ but $\chi_\Sigma\notin \mathcal{B}$. Since $\Sigma$ is $\beta$-saturated, $\Sigma\nvdash\beta$. This implies, by Lemma \ref{lem:TL_complete}, that $\Sigma\nvdash_\mathbb{SCS}\beta$. 

We claim that $\Sigma\vdash_\mathcal{B}\beta$. Suppose the contrary, i.e. $\Sigma\nvdash_\mathcal{B}\beta$. This implies that, there exists $\delta\in\mathscr{L}$ and a strongly closed $\delta$-saturated $\Delta\subseteq \mathscr{L}$ such that $\chi_\Delta\in\mathcal{B}$ with $\chi_\Delta(\Sigma)=\{1\}$ but $\chi_\Delta(\beta)\neq1$. This implies that $\Sigma\subseteq\Delta$ but $\beta\notin\Delta$. Now, since $\chi_\Delta\in\mathcal{B}$, we can conclude that $\Delta\neq\Sigma$, as $\chi_\Sigma\notin\mathcal{B}$. Thus $\Sigma\subsetneq\Delta$, which implies that there exists $\sigma\in \Delta\setminus\Sigma$. Then, as $\Sigma$ is $\beta$-saturated, $\Sigma\cup\{\sigma\}\vdash \beta$. Now, since $\Delta$ is strongly closed and $\Sigma\cup\{\sigma\}\subseteq \Delta$, $\Delta\vdash\beta$, and hence $\beta\in\Delta$. This is a contradiction. So, $\Sigma\vdash_\mathcal{B}\beta$ while $\Sigma\nvdash_\mathbb{SCS}\beta$, which implies that $\vdash_\mathbb{SCS}\,\subsetneq\,\vdash_\mathcal{B}$. Hence, by Theorem \ref{thm:Adeq_TL}, $\vdash\,\subsetneq\,\vdash_\mathcal{B}$.
\end{proof}

\begin{theorem}[\textsc{Representation Theorem for $\tl{4}$}]\label{thm:Rep_TL}
Let $(\mathscr{L},\vdash)$ be a logical structure. We define
\[
\begin{array}{c}
\mathbb{SCS}[\beta,\alpha]=\{\chi_\Sigma\in \mathbb{SCS}\mid\,\Sigma~\hbox{is}~\beta\hbox{-saturated},~\Sigma\nvdash\alpha~\hbox{and}~\{\beta\}\vdash\alpha\}\\
\hbox{and}\\
\mathbb{SCS}^\ast=\displaystyle\bigcup_{\alpha,\beta\in \mathscr{L}}\mathbb{SCS}[\beta,\alpha].
\end{array}
\]
Then, $(\mathscr{L},\vdash)$ is of $\tl{4}$-type iff $\vdash\,=\,\vdash_{\mathbb{SCS}^\ast}$. 
\end{theorem}

\begin{proof}
Suppose $(\mathscr{L},\vdash)$ is of $\tl{4}$-type. 
We note that $\mathbb{SCS}^*\subseteq \mathbb{SCS}$. Then, by Theorem \ref{thm:Min_TL}, $\vdash\,\subseteq\, \vdash_{\mathbb{SCS}^\ast}$. To prove the reverse inclusion, let $\Gamma\cup\{\alpha\}\subseteq\mathscr{L}$ such that $\Gamma\nvdash\alpha$. Then, since $(\mathscr{L},\vdash)$ is of $\tl{4}$-type, by statement (4) of  Theorem~\ref{thm:Char_TL4}, there exists $\beta\in \mathscr{L}$ and a strongly closed $\beta$-saturated $\Sigma\supseteq\Gamma$ such that $\Sigma\nvdash\alpha$ and $\{\beta\}\vdash\alpha$. This implies that $\chi_\Sigma\in \mathbb{SCS}[\beta,\alpha]\subseteq \mathbb{SCS}^\ast$. Now, as $\Gamma\subseteq\Sigma$, $\chi_\Sigma(\Gamma)=\{1\}$. However, since $\Sigma$ is strongly closed and $\Sigma\nvdash\alpha$, $\alpha\notin\Sigma$, which implies that $\chi_\Sigma(\alpha)=0$. This means that $\Gamma\nvdash_{\mathbb{SCS}^\ast}\alpha$. Thus, $\Gamma\nvdash\alpha$ implies $\Gamma\nvdash_{\mathbb{SCS}^\ast}\alpha$. So, we can conclude that $\vdash_{\mathbb{SCS}^\ast}\,\subseteq\,\vdash$. Hence $\vdash\,=\,\vdash_{\mathbb{SCS}^\ast}$.

Conversely, suppose $\vdash\,=\,\vdash_{\mathbb{SCS}^\ast}$.
Let $\Gamma\cup\{\alpha\}\subseteq \mathscr{L}$ be such that $\Gamma\nvdash\alpha$. Then, as $\vdash\,=\,\vdash_{\mathbb{SCS}^\ast}$, $\Gamma\nvdash_{\mathbb{SCS}^\ast}\alpha$. So, there exists a $\chi_\Sigma\in\mathbb{SCS}^\ast$, where $\Sigma$ is strongly closed, $\beta$-saturated for some $\beta\in\mathscr{L}$, $\Sigma\nvdash\alpha$ and $\{\beta\}\vdash\alpha$, such that $\chi_\Sigma(\alpha)=0$ but $\chi_\Sigma(\Gamma)=\{1\}$, i.e. $\Sigma\supseteq\Gamma$. Hence, by statement (4) of Theorem~\ref{thm:Char_TL4}, $(\mathscr{L},\vdash)$ is of $\tl{4}$-type.
\end{proof}

The following two results were proved in \cite{Beziau1999}\footnote{The presentations of these results here differ from those in \cite{Beziau1999}, mainly due to differing terminology. One can, however, easily prove that our formulations are equivalent to the original ones.}.
\begin{enumerate}[label=(\arabic*)]
\item Let $(\mathscr{L},\vdash)$ be a logical structure and 
\[
\mathbb{RELMAX}=\displaystyle\bigcup_{\alpha\in \mathscr{L}}\{\chi_\Sigma:\Sigma~\hbox{is relatively maximal in}~\alpha\}
\]
If $(\mathscr{L},\vdash)$ is of finitary Tarski-type then $\vdash_{\mathbb{RELMAX}}\,=\,\vdash$.
\item For a given finitary Tarski-type logical structure $(\mathscr{L},\vdash)$ and for any $\mathcal{B}\subsetneq \mathbb{RELMAX}$ we have  $\vdash\,\subsetneq\,\vdash_{\mathcal{B}}$.
\end{enumerate}
These can now be seen as special cases of Theorems \ref{thm:Adeq_TL} and \ref{thm:Min_TL} via the following observations.
\begin{enumerate}[label=(\alph*)]
    \item Every finitary Tarski-type logical structure is of $\tl{4}$-type, by Corollary~\ref{cor:finTarski=>LindIV}.
    \item Let $(\mathscr{L},\vdash)$ be a Tarski-type logical structure and $\Gamma\cup\{\alpha\}\subseteq\mathscr{L}$ be such that $\Gamma$ is relatively maximal in $\alpha$. Then $\Gamma$ is strongly closed.
    
\noindent To prove this, we note that since $\Gamma$ is relatively maximal in $\alpha$, $\Gamma\nvdash\alpha$. Then, since $(\mathscr{L},\vdash)$ is of Tarski-type, by Theorem~\ref{thm:Char_Tar}, there exists a strongly closed $\Sigma\supseteq\Gamma$ such that $\Sigma\nvdash\alpha$. Since $\Gamma$ is relatively maximal in $\alpha$, this implies $\Sigma=\Gamma$. Hence $\Gamma$ is strongly closed.
    \item Let $(\mathscr{L},\vdash)$ be a Tarski-type logical structure. Then, by (b) above, and Remark \ref{rem:alpha-sat}, for any $\Sigma\cup\{\alpha\}\subseteq\mathscr{L}$, $\Sigma$ is relatively maximal in $\alpha$ iff it is strongly closed and $\alpha$-saturated. Hence, for finitary Tarski-type logical structures, $\mathbb{RELMAX}=\mathbb{SCS}$.
\end{enumerate}

\section{Concluding Remarks}\label{sec:Concl}
The main contributions in this paper can be summarized as follows.
\begin{itemize}
    \item We have studied four classes of Lindenbaum-type logical structures and proved characterization theorems for two of them.
    \item A new characterization theorem for the previously known Tarski-type logical structures has been proved. We have also discussed a representation theorem for these logical structures using the notion of a Suszko set.
    \item The classes of Tarski-type and the four Lindenbaum-type logical structures have been separated using multiple examples.
    \item We have next studied the logical structures that are of, both Tarski- and a Lindenbaum-type. These have been named as the $\tl{i}$-type logical structures, where $i=1,2,3,4$. The $\tl{4}$-type logical structures have been studied in detail. We have proved characterization, adequacy, minimality, and representation theorems for these. It has been pointed out that these are generalizations of some results proved by B\'eziau in \cite{Beziau1999}, for finitary Tarski-type logical structures.
    \item It has been pointed out that the class of $\tl{2}$-type logical structures coincides with that of the $\tl{4}$-type ones. The rest have been shown to be separate using examples.
    \item The $\alpha$-saturated sets have also been studied and it has been shown that the notion offers a generalization to some better-known ones, viz., those of the maximal consistent sets in classical propositional logic, and the implication-saturated sets introduced by Batens in \cite{Batens1980}.
\end{itemize}

Finally, the following are some possible directions for future work.
\begin{itemize}
    \item Although the study of specific classes of logical structures is important, the spirit of universal logic lies in developing a general theory. This is embodied by results such as our Lemma \ref{lem:LS_soundness}. We expect to have more such results in future.
    \item The characterization theorems for Lindenbaum-I- and Lindenbaum-II-type logical structures remain as topics for future projects.
    \item In the current paper, we have provided a number of examples to separate the various classes of logical structures. It will be interesting to see if these examples can be characterized in some way. In particular, one might ask if the use of von Neumann ordinals is essential to separate the $\tl{1}$- from $\tl{2}$- and $\tl{3}$-type logical structures.
    \item It is expected that results similar to the ones proved in Section \ref{sec:TL-4} can be proved for $\tl{i}$-structures, where $i=1,2,3$. These are also left for the future.
    \item Lastly, one could work with more general sets instead of $\{0, 1\}$. Universal logic has deep connections with Suszko's Thesis and many-valued logics (see \cite{Beziau2020}). Investigation in these lines remains a job for the future.
\end{itemize}

\section*{Acknowledgements}
The authors wish to express their gratitude towards Prof.\ Peter Arndt, Lt.\ Prof.\ John Corcoran for their suggestions and encouragement, and two anonymous referees for their comments on the extended abstract of an earlier version of the paper submitted to the 9th Indian Conference on Logic and its Applications (ICLA), 2021.
 
\bibliographystyle{plain}
\bibliography{ref}
\end{document}